\documentclass[12pt]{amsart}
\usepackage{amssymb}

\usepackage{color}
\newtheorem{theorem}{Theorem}[section]
\newtheorem{lemma}[theorem]{Lemma}
\newtheorem{remark}[theorem]{Remark}
\newtheorem{definition}[theorem]{Definition}

\newtheorem{corollary}[theorem]{Corollary}
\newtheorem{proposition}[theorem]{Proposition}

\newtheorem{lem-def}[theorem]{Lemma-Definition}
\newtheorem{example}[theorem]{Example}
\DeclareRobustCommand\longtwoheadrightarrow
{\relbar\joinrel\twoheadrightarrow}

\renewenvironment{proof}{{\bfseries Proof.}}{\qed}
\newcommand{\ra}{\rightarrow}

\newcommand{\N}{\mathbb N}
\newcommand{\Z}{\mathbb Z}
\newcommand{\Q}{\mathbb Q}

\newcommand{\F}{\mathbb F}

\def\op{\operatorname}

\def\al{\alpha}

\def\ars#1{\renewcommand\arraystretch{#1}}

\def\aut{\op{Aut}}

\def\bs{\vskip.5cm}
\def\be{\beta}

\def\cc{\mathcal{C}}

\def\inn{\op{in}}

\def\d{\Delta}
\def\defn{\nn{\bf Definition. }}

\def\dgm{\op{deg}_\mu}
\def\dgmh{\op{deg}_{\muh}}

\def\dm{\Delta_\mu}

\def\diso{\lower.4ex\hbox{$\downarrow$}\raise.4ex\hbox{\mbox{\scriptsize
$\wr$}}}

\def\dta{\delta}

\def\e{\medskip}

\def\ee{\mathcal{E}}

\def\ep{\epsilon}
\def\epm{\epsilon_\mu}

\def\g{\Gamma}
\def\ga{\gamma}

\def\gal{\op{Gal}}

\def\gen#1{\big\langle\, {#1} \,\big\rangle}
\def\geta{\Gamma_\eta}

\def\gg{\mathcal{G}}

\def\ggm{\mathcal{G}_\mu}
\def\ggmh{\mathcal{G}_{\muh}}

\def\ggnh{\mathcal{G}_{\nuh}}

\def\gm{\g_\mu}

\def\gn{\g_\nu}

\def\gq{\g_\Q}
\def\gr{\operatorname{gr}}

\def\hk{\hookrightarrow}

\def\kth{K_\tau^h}
\def\kthx{K_\tau^h[x]}

\def\im{\op{Im}}

\def\inm{\op{in}_\mu}
\def\inmh{\op{in}_{\muh}}
\def\inn{\op{in}}

\def\irr{\op{Irr}}
\def\ism{\lower.3ex\hbox{\ars{.08}$\begin{array}{c}\,\to\\\mbox{\tiny $\sim\,$}\end{array}$}}
\def\iso{\ \lower.3ex\hbox{\ars{.08}$\begin{array}{c}\lra\\\mbox{\tiny $\sim\,$}\end{array}$}\ }

\def\ka{\kappa}

\def\kb{\overline{K}}
\def\vt{\vb_\tau}
\def\kh{K^h}
\def\khx{K^h[x]}
\def\km{k_\mu}

\def\kp{\op{KP}}

\def\kpi{\op{KP}_\infty}
\def\kpm{\op{KP}(\mu)}

\def\kpn{\op{KP}(\nu)}

\def\ks{K^{\op{sep}}}
\def\kxb{\overline{K(x)}}

\def\kx{K[x]}

\def\la{\lambda}
\def\La{\Lambda}

\def\lg{l\raise.6ex\hbox to.2em{\hss.\hss}l}

\def\lra{\,\longrightarrow\,}

\def\lx{\operatorname{lex}}

\def\md#1{\; \mbox{\rm(mod }{#1})}

\def\mon{\op{Mon}(L/K)}

\def\muh{\mu^h}
\def\mul{\mu_\la}

\def\nn{\noindent}

\def\np{N_{\mu,\phi}}
\def\npp{N^+_{\mu,\phi}}

\def\nuh{\nu^h}

\def\om{\omega}

\def\orb{\hbox to  .3em{$\backslash$}\backslash}
\def\ord{\op{ord}}
\def\mub{\overline{\mu}}
\def\p{\mathfrak{p}}

\def\pp{\mathfrak{P}}
\def\ppa{\mathcal{P}_{\al}}

\def\pset{\mathcal{P}}

\def\ri{\rho_i}

\def\rj{\rho_j}

\def\kx{K[x]}

\def\sg{\sigma}

\def\sii{\ \Longleftrightarrow\ }

\def\smu{\sim_\mu}

\def\sp{\op{Spec}}

\def\supp{\op{supp}}

\def\t{\theta}

\def\dm{\Delta_\mu}
\def\dn{\Delta_\nu}

\def\ttt{\mathcal{T}}
\def\ty{\mathbf{t}}

\def\vb{\bar{v}}
\def\vh{v^h}

\def\ab{\overline{a}}

\def\wb{\overline{w}}

\def\omb{\overline{\omega}}

\DeclareMathOperator{\inv}{in}
\newcounter{cs}
\stepcounter{cs}
\newcommand{\casos}{\begin{itemize}}
\newcommand{\fcasos}{\end{itemize}\setcounter{cs}{1}}

\newfont{\tit}{cmr12 scaled \magstep3}

\title{The defect formula}
\author{Enric Nart}
\address{Departament de Matem\`{a}tiques,         Universitat Aut\`{o}noma de Barcelona,         Edifici C, E-08193 Bellaterra, Barcelona, Catalonia}
\email{nart@mat.uab.cat}
\author{Josnei Novacoski}
\address{Departamento de Matem\'{a}tica,         Universidade Federal de S\~ao Carlos, Rod. Washington Luís, 235, 13565--905, S\~ao Carlos -SP, Brazil}
\email{josnei@ufscar.br}
\thanks{Partially supported by grant PID2020-116542GB-I00  funded by the Spanish MCIN/AEI. 
During the realization of this project the second author was supported by two grants from Funda\c{c}\~ao de Amparo \`a Pesquisa do Estado de S\~ao Paulo (process numbers 2017/17835-9 and 2021/11246-7).}

\begin{document}

\begin{abstract}
In this paper we present a characterization for the defect of a simple algebraic extensions of valued fields. This characterization generalizes the known result for the henselian case, namely that the defect is the product of the relative degrees of limit augmentations. The main tool used here is the graded algebra associated to a valuation on a polynomial ring. Let $\kh$ be a henselization of a valued field $K$. Another relevant result proved in this paper is that for every valuation $\muh$ on $\khx$, with restriction $\mu$ on $\kx$, the corresponding map $\mathcal G_\mu\hk\mathcal G_{\muh}$ of graded algebras is an isomorphism.
\end{abstract}

\maketitle

\section{Introduction}

Let $(K,v)$ be a valued field and $w$ an extension of $v$ to a finite extension $L$ of $K$. If $w$ is the unique extension of $v$ to $L$, then we define the defect of $w/v$ by
\[
d(w/v):=\frac{[L:K]}{e\cdot f},
\]
where $e$ and $f$ are the ramification index and inertia degree  of $w/v$, respectively. 
Fix an algebraic closure $\overline{K}$ of $K$ and an extension $\vb$ of $w$ to $\kb$. We consider the henselizations $K^h$ and $L^h$ of $K$ and $L$ inside this algebraic closure. Set $\vh=\vb_{\mid \kh}$ and $w^h=\vb_{\mid L^h}$. If $w$ is not the unique extension of $K$ to $L$, then we define
\[
d(w/v):=d(w^h/v^h).
\]
We will show (Section \ref{independencehens}) that this number does not depend on the choice of $\overline v$.

The main goal of this paper is to understand the defect $d(w/v)$ of an extension of $v$ to a simple finite extension $L=K[x]/(g)$, defined by an irreducible polynomial $g\in \kx$. The valuation $w$ may be identified with a valuation $\mu$ on $K[x]$ with
nontrivial support $gK[x]$. We use the Mac Lane-Vaqui\'e machinery  to compute  $d(w/v)$ in terms of data associated to the valuation $\mu$.

Let $\mu\ra \nu$ be an augmentation (ordinary or limit) of valuations on $K[x]$. Let $\phi\in\kx$ be either a key polynomial of minimal degree of $\nu$, or $\supp(\nu)=\phi\kx$. In Section \ref{defectofaaug} we define the defect $d(\mu\ra\nu)$ of $\mu\ra \nu$. This number is the degree of the image of $\phi$ in a suitable graded algebra. We show that if $\mu\ra \nu$ is ordinary, then $d(\mu\ra\nu)=1$. We also show that if $(K,v)$ is henselian and $\mu\ra \nu$ is a limit augmentation, then $d(\mu\ra\nu)$ is equal to $\deg(\nu)/\deg(\mu)$ (Lemma \ref{Lemma3}).

By a celebrated result of Mac Lane--Vaqui\'e, $\mu$ may be obtained by a finite chain of augmentations
\[
v\ \to\  \mu_0\ \to\  \mu_1\ \to\ \cdots
\ \to\ \mu_{r-1} 
\ \to\ \mu_{r}=\mu
\]
starting with $v$. 

Our main result (Theorem \ref{DF}) shows that, under certain mild assumptions on this chain, we have
\[
 d(w/v)=d(\mu_0\to\mu_1)\cdots d(\mu_{r-1}\to\mu). 
\]
If $(K,v)$ is henselian, then our formula becomes
\[
d(w/v)=\prod_{i\in I}\frac{\deg(\mu_{i+1})}{\deg(\mu_{i})}
\]
where $I$ is the set of indices $i$ for which $\mu_i\to\mu_{i+1}$ is a limit augmentation. This is precisely Vaqui\'e's formula, proved in \cite{Vaq3}.

Our proof relies heavily in an analysis of the behavior of all ingredients under henselization. The main goal is to prove the following theorem.
\begin{theorem}\label{defectcte}
Let $\mu\ra \nu$ be an augmentation of valuations on $K[x]$ and $\mu^h\ra \nu^h$ the corresponding augmentation on $K^h[x]$. Then
\[
d(\mu\to\nu)=d(\muh\to\nu^h).
\]
\end{theorem}

The proof of Theorem \ref{defectcte} is based on two main ingredients. The first main ingredient (Theorem \ref{fundamental}) is about properties of irreducible polynomials over henselian fields. This result was proved in \cite[Section 4]{defless} for \emph{defectless} irreducible polynomials over henselian fields. With slight modifications, the proof (which we present here for the ease of the reader) is valid for arbitrary irreducible polynomials.

The second main ingredient is about the structure of the graded algebras associated to the corresponding valuations. For a field extension $L/K$, if $\mu$ is a valuation on $K[x]$ and $\omega$ is an extension of $\mu$ to $L[x]$, then the canonical map $\iota: \mathcal G_\mu\hk \mathcal G_\omega$ is injective. It is natural to ask when it is an isomorphism. We will show the following result.
\begin{theorem}\label{theoremabouthensel}
Let $\mu$ be a valuation on $K[x]$ and $\mu^h$ its extension to $K^h[x]$. Then the mapping $\iota: \mathcal G_\mu\hk \mathcal G_{\mu^h}$ is an isomorphism.
\end{theorem}

Fix the notation of Theorem \ref{defectcte} and assume that $\mu\to\nu$ is limit. Then,
\[
d(\mu\to\nu)=d(\muh\to\nu^h)=\deg(\nu^h)/\deg(\mu^h),
\]
but  (see Example \ref{exa1}) it is not necessarily true that 
\[
d(\mu\to\nu)=\deg(\nu)/\deg(\mu).                                                       
\]
However, if $\mu$ and $\nu$ are rank one valuations and $\supp(\nu)=0$, then the latter equality is also satisfied (see Section \ref{rankone}).

A related question is whether key polynomials for valuations on $K[x]$ are irreducible in $K^h[x]$. This question appears in \cite{Andrei} where the author studies \textit{diskoids} in order to understand extensions of valuations. Example \ref{exampnaphel} gives a negative answer to this question. We also provided a criterion (Corollary \ref{corollaryandrei}) that characterizes key polynomials for valuations on $K[x]$ that are irreducible in $K^h[x]$. In particular, it follows from the rank one analysis that this is true for rank one valuations with trivial support.

This paper is divided as follows. In Section \ref{preliminaries} we present the main definitions and results related to graded algebras that will be used in the paper. In Section \ref{secFinLeaves} we present some results about extensions of valuations to henselizations. Section \ref{newtonpoly} is devoted to introduce Newton polygons as a tool to prove Theorem \ref{fundamental}, which is a kind of generalization of Hensel's lemma. In Section \ref{prooofthm1} we present, among other things, the proof of Theorem \ref{theoremabouthensel}. The main purpose of Section \ref{Defectform} is to introduce the defect of an augmentation and relate it with the classic definition of defect. Finally, in Section \ref{rigofdefect} we present the proof of our main theorem (Theorem \ref{defectcte}). In Section \ref{rankone}, we discuss the rank one case. The results of this paper have a special form in this case. Finally, in the last section we present examples that illustrate the situations appearing in the paper.

\section{Graded algebra of a valuation on a polynomial ring}\label{preliminaries}

Let $(K,v)$ be a valued field with value group $\g=vK$ and residue field $k=Kv$.

Let $\mu$ be a valuation on $\kx$ extending $v$. The \textbf{support} of $\mu$ is the prime ideal 
\[
\supp(\mu)=\mu^{-1}(\infty)\in\sp(\kx).
\]
Only the valuations with trivial support extend to a valuation on the field $K(x)$.

The \textbf{value group} of $\mu$ is the subgroup $\gm$ generated by $\mu\left(\kx\setminus\supp(\mu)\right)$. The \textbf{residue field} $\km$ is defined as the residue field of the valuation naturally induced by $\mu$ on the field  of fractions of  $\kx/\supp(\mu)$.    

Any extension of $v$ to $K[x]$ is (exclusively) of one of the following types:
\begin{itemize}
	\item\textbf{Nontrivial support}: $\supp(\mu)=fK[x]$ for some $f\neq 0$.
	\item\textbf{Value-transcendental}: $\supp(\mu)=0$ and $\gm/\g$ is not a torsion group.
	\item\textbf{Residue-transcendental}: $\supp(\mu)=0$ and the extension $\km/k$  is transcendental.
	\item\textbf{Valuation-algebraic}: $\supp(\mu)=0$,  $\gm/\g$ is a torsion group and the extension $\km/k$  is algebraic.
\end{itemize}
The valuation $\mu$ will be called \textbf{valuation-transcendental} if it is value-transcendental or residue-transcendental.

For any field $\mathbb{K}$, we let $\irr(\mathbb{K})$ be the set of monic irreducible polynomials in $\mathbb{K}[x]$.

\subsection{Key polynomials}\label{subsecKP}

Let $\mu$ be an extension of $v$ to $\kx$. For all $\alpha\in\Gamma_\mu$ we consider the abelian groups
\[
\mathcal P_\alpha:=\{f\in K[x]\mid \mu(f)\geq \al\}\mbox{ and }\mathcal P_\alpha^+:=\{f\in K[x]\mid \mu(f)> \al\}.
\]
We define the {\bf graded algebra} of $\mu$ by
\[
\mathcal G_\mu:=\bigoplus_{\alpha \in \Gamma_\mu}\mathcal P_\alpha/\mathcal P_\alpha^+.
\]

For $f\in K[x]\setminus \supp(\mu)$, we denote its image in $\mathcal P_{\mu(f)}/\mathcal P_{\mu(f)}^+\subset \mathcal{G}_\mu$ by $\inv_\mu f$. 

For $f,g\in K[x]\setminus \supp(\mu)$ we say that $f$ and $g$ are \textbf{$\mu$-equivalent} if $\inv_\mu f=\inv_\mu g$.
In this case, we write $f\sim_\mu g$.

\begin{definition}
	For a valuation $\mu$ on $K[x]$ we say that a monic polynomial $\phi\in K[x]$ is a \textbf{Mac Lane--Vaqui\'e key polynomial} (or only a key polynomial) for $\mu$ if
	\begin{itemize}
		\item\textit{$\phi$ is $\mu$-irreducible}: if $\inv_\mu\phi\mid \inv_\mu(f\cdot g)$, then $\inv_\mu\phi\mid \inv_\mu f$ or $\inv_\mu \phi\mid \inv_\mu g$; and
		\item \textit{$\phi$ is $\mu$-minimal}: if $\inv_\mu\phi\mid \inv_\mu f$ and $f\ne0$, then $\deg(\phi)\leq \deg(f)$.
	\end{itemize}
\end{definition}
We denote by ${\rm KP}(\mu)$ the set of all key polynomials for $\mu$. They are necessarily irreducible in $\kx$.
The equivalence relation $\smu$ restricts to an equivalence relation on the set $\kpm$. For all $\phi\in\kpm$ we denote its class by
\[
[\phi]_\mu=\{\varphi\in\kpm\mid \phi\smu\varphi\}.
\]
All polynomials in $[\phi]_\mu$ have the same degree \cite[Proposition 6.6]{KP}. We denote by $\deg[\phi]_\mu$ this common degree.

It is easy to characterize the existence of key polynomials.
Let $\ggm^0\subset\ggm$ be the subalgebra generated by the set of all homogeneous units.

\begin{theorem}\cite[Theorem 4.4]{KP}\label{empty}
	For every valuation $\mu$ on $\kx$, the following conditions are equivalent.
	\begin{enumerate}
		\item[(a)] $\kpm=\emptyset$.
		\item[(b)] The valuation $\mu$ has nontrivial support, or it is valuation-algebraic.
		\item[(c)] The algebra $\ggm$ is simple; that is, $\ggm^0=\ggm$. 
	\end{enumerate}
	Therefore, $\kpm\ne\emptyset$ if and only if $\mu$ is valuation-transcendental. 
\end{theorem}

\begin{definition}
	If ${\rm KP}(\mu)\neq \emptyset$, then we define
	$\deg(\mu)=\min_{\phi\in {\rm KP}(\mu)} \deg(\phi)$.
	
	If $\supp(\mu)=f\kx$ for some $f\in\irr(K)$, then we define $\deg(\mu)=\deg(f)$.
\end{definition}

By Theorem \ref{empty}, these definitions of $\deg(\mu)$ are completely independent.  

\subsection{Tangent directions}
Let $L/K$ be a field extension, $\mu$ a valuation on $K[x]$ and $\eta$ a valuation on $L[x]$, taking values in a common ordered abelian group. We say that $\mu\leq \eta$ if $\mu(f)\leq \eta(f)$ for all $f\in K[x]$. In this situation, we define the map
\begin{displaymath}
\begin{array}{ccc}
\mathcal G_\mu&\lra& \mathcal G_\eta\\[8pt]
\inv_\mu(f)&\longmapsto &
\left\{\begin{array}{cl}
\inv_\eta(f)&\mbox{ if }\mu(f)=\eta(f)\\
0 & \mbox{ if }\mu(f)<\eta(f)
\end{array}
\right..
\end{array}
\end{displaymath}
\begin{remark}
	\begin{description}
		\item[(i)] If $\eta$ is an extension of $\mu$ (i.e., $\mu(f)=\eta(f)$ for all $f\in K[x]$), then $\mathcal G_\mu\ra\mathcal G_\eta$ is injective.
		\item[(ii)] If $K=L$, then $\mathcal G_\mu\ra\mathcal G_\eta$ is injective if and only if $\mu=\eta$.
	\end{description}
\end{remark}

\begin{definition}
	Let $\mu, \eta$ be two valuations on $K[x]$ such that $\mu<\eta$ (i.e., $\mu\leq \eta$ and $\mu\neq \eta$). Let $\ty(\mu,\eta)$ be the set of monic polynomials $\phi\in K[x]$ of minimal degree satisfying $\mu(\phi)<\eta(\phi)$.
	We say that $\ty(\mu,\eta)$ is the\textbf{ tangent direction of $\mu$}, determined by $\eta$. 
\end{definition}

The following result follows from \cite[Theorem 1.15]{Vaq} and \cite[Corollary 2.5]{MLV}.

\begin{lemma}\label{tdef}
	Every $\phi\in \ty(\mu,\eta)$ is a key polynomial for $\mu$ and $\ty(\mu,\eta)=[\phi]_\mu$.
	Moreover, for all nonzero $f\in\kx$, we have $\mu(f)<\eta(f)$ if and only if $\inm \phi\mid \inm f$.
\end{lemma}

\begin{lemma}\label{td}
	Let $\mu$ be a valuation on $\kx$ admitting a key polynomial $\phi\in\kpm$. For any valuation $\mu<\eta$, we have
	$$
	\ty(\mu,\eta)=[\phi]_\mu\ \sii\ \mu(\phi)<\eta(\phi).
	$$
\end{lemma}

\begin{proof}
	If $\ty(\mu,\eta)=[\phi]_\mu$, then $\mu(\phi)<\eta(\phi)$ by the definition of the tangent direction.
	
	Conversely, suppose $\mu(\phi)<\eta(\phi)$ and let $\ty(\mu,\eta)=[\varphi]_\mu$. Then, $\inm \varphi\mid\inm\phi$, and this implies $\varphi\smu\phi$ by \cite[Proposition 6.6]{KP}.  
\end{proof}\bs

\subsection{Degree of homogeneous elements and relative ramification index}\label{subsecDegree}

Let $\phi\in\kpm$. For any $f\in K[x]$ there exist uniquely determined $f_0,\ldots,f_r\in K[x]$ 
with $\deg(f_i)<\deg(\phi)$ for every $i$, $0\leq i\leq r$, such that
\[%\begin{equation}\label{phiexpansion}
f=f_0+f_1\phi+\ldots+f_r\phi^r.
\]%\end{equation}
This expression is called the \textbf{$\phi$-expansion of $f$}.

Since $\phi$ is $\mu$-minimal, we have
\begin{equation}\label{minimal}
\mu(f)=\min_{0\leq i\leq r}\left\{\mu(f_i\phi^i)\right\}.
\end{equation}
Consider the set $S=S_{\mu,\phi}(f)=\{i\mid \mu(f)=\mu(f_i\phi^i)\}$. Clearly,
$$%\begin{equation}\label{inmuf}
\inm f=\sum_{i\in S}(\inm f_i)\pi^i,\quad \pi=\inm\phi.
$$%\end{equation}

\emph{Suppose that $\phi$ has minimal degree in $\kpm$.} In this case, all these coefficients $\inm f_i$ are homogeneous units \cite[Proposition 3.5]{KP}. Thus, $\inm f$ belongs to the subalgebra $\ggm^0[\pi]$. In particular, this proves that $\ggm^0[\pi]=\ggm$.

The representation of $\inm f$ as a polynomial in $\ggm^0[\pi]$ is unique, as the following result shows (cf. \cite[Remark 16]{Dec} or \cite[Proposition 4.5]{N2021}). 

\begin{theorem}\label{g0gm}
	Suppose that  $\kpm\ne\emptyset$ and let $\phi$ be a key polynomial of minimal degree for $\mu$. 
	Then, the prime $\pi=\inm\phi$ is transcendental over $\ggm^0$ and
	$\ggm=\ggm^0[\pi]$.
\end{theorem}

\begin{definition}
For all nonzero $f\in\kx$ we define its $\mu$-\textbf{degree} $\dgm(f)\in\N$  as the degree  of $\inm f$ as a polynomial in $\inm \phi$ with coefficients in $\ggm^0$, for some $\phi\in\kpm$ of minimal degree. 
\end{definition}

These definitions are independent of the choice of $\phi$ among all key poynomials of minimal degree for $\mu$.   
Note that a homogeneous element $\inm f$ is a unit if and only if $\dgm(f)=0$. Also, we have in general:  $\dgm(f)=\max\left(S_{\mu,\phi}(f)\right)$.\e

\nn{\bf Convention. }If $\kpm=\emptyset$, then $\ggm^0=\ggm$ by Theorem \ref{empty}. In this case, we agree that $\dgm(f)=0$ for all $f\in\kx\setminus\supp(\mu)$.

\begin{lemma}\label{lemma1x}\cite[Proposition 3.7]{KP}
	Suppose that  $\kpm\ne\emptyset$. A nonzero $f\in\kx$ is $\mu$-minimal if and only if $\deg(f)=\deg_\mu(f)\deg(\mu)$. 
\end{lemma}

Caution! The $\mu$-degree $\dgm(f)$ should not be confused with the  
natural \textbf{grade}: $$\al=\gr_\mu(\inm f)=\mu(f)\in\gm$$ which indicates that $\inm f$ belongs to the homogeneous component $\ppa/\ppa^+$.

\begin{definition}
For any valuation $\mu$ on $\kx$, let 
$\gm^0\subset\gm$ be the subgroup of grades of all homogeneous units in $\ggm$.
The \textbf{relative ramification index} of $\mu$ is defined as $$e=e(\mu)=\left(\gm\colon \gm^0\right).$$
\end{definition}

The group $\gm^0/\g$ is a torsion group, because the homogeneous units are algebraic over the graded algebra of $v$ \cite[Proposition 3.5]{KP}.
On the other hand, if $\phi$ is a key polynomial for $\mu$  of minimal degree, then 
$$ \gm^0=\left\{\mu(a)\mid a\in \kx,\ 0\le \deg(a)<\deg(\phi)\right\},$$
and $\gm=\gen{\gm^0,\mu(\phi)}$ by (\ref{minimal}). The following statements follow immediately from these facts and Theorem \ref{empty}. 

\begin{itemize}
\item 	If $\mu$ is value-transcendental, then $e=\infty$.
\item If $\mu$ is residue-transcendental and $\phi$ is a key polynomial for $\mu$  of minimal degree, then  $e$ is the least positive integer such that $e\mu(\phi)\in\gm^0$.
\item If $\kpm=\emptyset$, then $e=1$. 
\end{itemize}

\begin{lemma}\label{00}
If $\mu<\eta$, then $\gm^0\subset \Gamma_\eta^0$.
\end{lemma}

\begin{proof}
Let $\al=\mu(f)$ for some $f\in\kx\setminus\supp(\mu)$ such that $\inm f$ is a unit. The canonical homomorphism of graded algebras $\ggm\to\mathcal G_\eta$ maps homogeneous units to homogeneous units. Hence, $\mu(f)=\eta(f)$ and $\inm f$ is mapped to $\inv_\eta f$. In particular, $\al=\eta(f)$ belongs to $\Gamma_\eta^0$.  
\end{proof}

\subsection{Ordinary and limit augmentations}\label{subsecLimAug}
As shown in \cite[Theorem 4.2]{mcla}, we can construct new valuations on $\kx$ by defining their action on $\phi$-expansions.

\begin{lemma}\label{Mlvord}
If $\phi\in\kpm$ and $\gamma>\mu(\phi)$, then the map acting on $\phi$-expansions as
	\[
	\nu\left(\sum_{i\geq 0}a_i\phi^i\right)=\min_{i\geq 0}\{\mu(a_i)+i\gamma\}
	\]
is a valuation on $K[x]$ such that $\mu<\nu$. Moreover, $\ty(\mu,\nu)=[\phi]_\mu$.
\end{lemma}
\begin{definition}
	The valuation $\nu$ above is called an \textbf{ordinary augmentation} of $\mu$. We denote it by $\nu=[\mu;\phi,\gamma]$.   
\end{definition}

This augmented valuation has trivial support if $\ga<\infty$, while $\supp(\nu)=\phi\kx$ if $\ga=\infty$. In both cases, $\deg(\nu)=\deg(\phi)$. Indeed, $\phi$ becomes a key polynomial of minimal degree of $\nu$, if $\ga<\infty$ \cite[Corollary 7.3]{KP}.

Let $A$ be a well-ordered set without a last element. A family $\cc=\left(\ri\right)_{i\in A}$ of valuations on $\kx$ is said to be a \textbf{continuous family}, parametrized by $A$, if:
\begin{itemize}
	\item The map $i\mapsto \rho_i$ is an isomorphism of totally ordered sets.
	\item The set $\{\deg(\ri)\}_{i\in A}$ is stable (i.e., there exists $i_0\in A$ such that $\deg(\ri)=\deg(\rho_{i_0})$ for all $i\geq i_0$). We denote by $\deg(\cc)$ this stable degree. 
\end{itemize}

A polynomial $f\in K[x]$ is said to be $\cc$-\textbf{stable} if there exists $i_0\in A$ such that $\rho_i(f)=\rho_{i_0}(f)$ for all $i\geq i_0$. In this case, we denote $\rho_\cc(f):=\rho_{i_0}(f)$.

\begin{definition}
	A monic polynomial $\phi\in\kx$ is a \textbf{limit key polynomial} for the family $\cc$ if it is $\cc$-unstable and has the smallest degree among $\cc$-unstable polynomials. Denote by ${\rm KP}_\infty(\cc)$ the set of all limit key polynomials for $\cc$.
\end{definition} 

If ${\rm KP}_\infty(\cc)=\emptyset$, then $\rho_\cc$ defines a valuation on $\kx$. We write $\rho_\cc=\lim(\cc)$ and we say that $\rho_\cc$ is the \textbf{stable limit} of the family $\cc$. 

Stable limits are valuation-algebraic \cite[Proposition 3.1]{MLV}.

If $\kpi(\cc)\ne\emptyset$, then a limit key polynomial can be used to augment the valuation as in the ordinary case (c.f. \cite[Proposition 1.22]{Vaq} or \cite[Theorem 5.16]{N2021}).

\begin{lemma}
	Suppose that $\phi$ is a limit key polynomial for $\cc$ and take $\gamma>\ri(\phi)$ for all $i\in A$. Then the  map acting on $\phi$-expansions as
	\[
	\nu\left(\sum_{i\geq 0}a_i\phi^i\right)=\min_{i\geq 0}\{\rho_\cc(a_i)+i\gamma\}
	\]
	is a valuation on $K[x]$ such that $\ri<\nu$ for all $i\in A$. 
\end{lemma}

\begin{definition}
	The valuation $\nu$ as above is called a \textbf{limit augmentation} of $\mu$. We denote it by $\nu=[\cc;\phi,\gamma]$.
\end{definition}

This augmented valuation has trivial support if $\ga<\infty$, while $\supp(\nu)=\phi\kx$ if $\ga=\infty$. In both cases, $\deg(\nu)=\deg(\phi)$. Indeed, $\phi$ becomes a key polynomial of minimal degree of $\nu$, if $\ga<\infty$ \cite[Corollary 7.13]{KP}.

From now on we will use the notation $\mu\,\ra\, \nu$
to describe an augmentation. This means that, either
\[
\ars{1.2}
\begin{array}{lll}
\nu=[\mu;\phi,\gamma], &\quad \phi\in {\rm KP}(\mu), &\quad \gamma>\mu(\phi), \ \mbox{ or}\\
\nu=[\cc;\phi,\gamma], &\quad \phi\in {\rm KP}_\infty(\cc), &\quad \gamma>\ri(\phi)\ \mbox{ for all }\ i\in A,
\end{array}
\]
where $\cc=\left(\ri\right)_{i\in A}$ is a continuous family  such that  $\mu=\rho_{i_{\op{min}}}$, for $i_{\op{min}}=\min(A)$.

We say that $\mu\ra \nu$ is an \emph{ordinary} or \emph{limit} augmentation, respectively. For the limit augmentations we have
\[
\ty(\mu,\nu)=\ty(\mu,\ri)\quad\mbox{  for all }i\in A.
\]

\begin{definition}
Let $\cc=\left(\ri\right)_{i\in A}$ be a continuous family  and let  $\mu=\rho_{i_{\op{min}}}$, for $i_{\op{min}}=\min(A)$.
We say that $\cc$ is \textbf{based on} $\mu$ if $\deg(\mu)=\deg(\cc)$.
\end{definition}

If $\mu\to\nu$ is a limit augmentation such that $\deg(\mu)<\deg(\cc)$,  then it splits into a chain of two augmentations
\[
\mu\to\rho_{i_0}\to\nu,
\]
where $i_0\in A$ satisfies $\deg(\rho_{i_0})=\deg(\cc)$.
The augmentation $\mu\to\rho_{i_0}$ is ordinary and $\rho_{i_0}\to\nu$ is a limit augmentation with respect to the continuous family $\cc'=\left(\ri\right)_{i\ge i_0}$, which is based on $\rho_{i_0}$.

Therefore, we may focus our analysis of limit augmentations on continuous families which are based on its initial valuation.

Moreover, since we are only interested in the limit behavior of continuous families, we may replace them by equivalent families  (both families cofinal in each other) satisfying certain  extra conditions described in \cite[Lemma 4.11]{VT}.\e

\nn{\bf Convention. }{\it In this paper, when referring to a continuous family $\cc=\left(\ri\right)_{i\in A}$, we will assume that 
\begin{itemize}
\item $\cc$  is based on its initial valuation.
\item All $\ri$ are residue-transcendental, have relative ramification index $e(\ri)=1$ and have a common value group: 
$$\g_\cc:=\g_{\ri}^0=\g_{\ri}\quad\mbox{for all }i\in A.$$
\end{itemize}	
}

There are augmentations which are simultaneously ordinary and limit. This occurs precisely when $\deg\kpi(\cc)=\deg(\cc)$. In this case, every $\phi\in\kpi(\cc)$ is a key polynomial for $\mu$ and $[\cc;\phi,\ga]=[\mu;\phi,\ga]$ whenever $\ga>\ri(\phi)$.\bs

\section{Extension of valuations and henselization}\label{secFinLeaves}

Let $(K,v)$ be a valued field with value group $\g=v(K^*)$. Let $\gq$ be the divisible closure of $\g$.

We fix an algebraic closure $\kb$ of $K$, and an extension $\vb$ of the valuation $v$ to $\kb$. This determines a henselization $(\kh,\vh)$ of $(K,v)$. If $\ks$ is the separable closure of $K$ in $\kb$, then the field $\kh$ is the fixed field of the \textbf{decomposition group} 
$$
D_{\vb}=\left\{\sg\in \gal(\ks/K)\mid \vb\circ\sg=\vb\right\}.
$$
The valuation $\vh$ is just the restriction of $\vb$ to $\kh$.
We have a chain of fields
$$
K\subset\kh\subset\ks\subset\kb.
$$
Since $\kb/\ks$ is purely inseparable, for all $\sg\in\aut(\kb/K)$ we have
$$
\vb\circ\sg=\vb\ \sii\ \sg\in\aut(\kb/\kh). 
$$ 
Thus, the valuation $\vh$ has a unique extension to $\kb$.\e

Let $g\in\irr(K)$. The extensions of $v$ to the simple field extension $\kx/(g)$ are in 1-1  correspondence with the irreducible factors of $g$ over $\khx$. For $v$ discrete of rank one, this fact goes back to Hensel. For an arbitrary valuation, it may be deduced from the techniques of \cite[Section 17]{endler}. For the ease of the reader, we provide a proof.

\subsection{Extensions of valuations to algebraic extensions}
Let us first describe all the extensions of $v$ to an arbitrary algebraic extension $L/K$. For any such extension $w/v$ the quotient $\g_w/\g$ of their value groups is a torsion group. Hence, we may assume that the extensions take values in the divisible closure $\gq$.

Thus, we aim to describe the set:
$$
\ee(L)=\left\{w\colon L\lra \gq\infty\,\mid\, w\, \mbox{ is a valuation extending }v\right\}.
$$

On the set  $\mon$ of all $K$-morphisms from $L$ to $\kb$, we define the following equivalence relation
$$
\la\sim_{\kh}\la' \ \sii\ \la'=\sg\circ\la\quad\mbox{for some}\quad \sg\in\aut(\kb/\kh).
$$
%We denote by $\mon{\kh}$ the quotient set $\mon{}/\!\sim_{\kh}$.

\begin{theorem}\label{endler}
	The  mapping $\mon\to\ee(L)$, defined by $\la\mapsto \vb\circ\la$, 
	induces a bijection between the quotient set $\mon{}/\!\sim_{\kh}$ and $\ee(L)$.
\end{theorem}

\begin{center}
	\setlength{\unitlength}{5mm}
	\begin{picture}(8,10)
	\put(4,0){$K$}\put(3.6,4){$\la(L)$}\put(4,8){$\kb$}
	\put(0,2){$L$}\put(8,4.3){$\kh$}\put(8,6.8){$\ks$}
	\put(4.4,1.1){\line(0,1){2.3}}\put(4.4,5.1){\line(0,1){2.4}}\put(8.4,5.2){\line(0,1){1.3}}
	\put(3.8,0.4){\vector(-2,1){2.9}}\put(1,2.5){\vector(2,1){2.5}}
	\put(5,.3){\line(4,5){2.9}}\put(5,8.2){\line(3,-1){2.7}}
	\put(1.6,3.3){\footnotesize{$\la$}}
	\end{picture}
\end{center}\bs

\begin{proof}
Let us first show that the mapping  $\mon\to\ee(L)$ is onto.

Take any $w\in\ee(L)$. For an arbitrary $\la\in\mon$, consider the following valuation on the field $\la(L)$: 
$$
w'=w\circ \la^{-1}.
$$
By Chevalley's theorem, there exists an extension $\vb'$ of $w'$ to $\kb$. By the Conjugation theorem \cite[Theorem 14.2]{endler}, there exists $\sg\in\aut(\kb/K)$ such that $\vb'=\vb\circ \sg$.

Consider the embedding $\la'=\sg\circ\la\in\mon$. The image of this embedding in $\ee(L)$ is
$$
\vb\circ \la'=\vb'\circ\sg^{-1}\circ\sg\circ\la=\vb'\circ \la =w.
$$
This shows that $\mon\to\ee(L)$ is onto.
To end the proof, we must check that
$$
\la\sim_{\kh} \la'\ \sii\ \vb\circ \la=\vb\circ \la'.
$$
The implication $\Rightarrow$ is obvious. If
$\la'=\sg\circ\la$ for some $\sg\in\aut(\kh/K)$, then
 $\vb\circ \la'=\vb\circ \sg\circ \la=\vb\circ\la$.
 
Conversely, suppose that $\vb\circ \la=\vb\circ \la'$.
Let $\tau\in\aut(\kb/K)$ be any lifting of the $K$-isomorphism
$$
\la'\circ\la^{-1}\colon \la(L)\lra \la'(L).
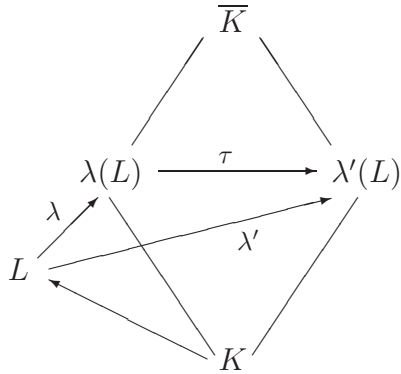
$$

\begin{center}
	\setlength{\unitlength}{5mm}
	\begin{picture}(8,11)
	\put(4,0){$K$}\put(0.3,5){$\la(L)$}\put(4,9){$\kb$}
	\put(-1.6,2.4){$L$}\put(7,5){$\la'(L)$}
	\put(7,6){\line(-2,3){1.9}}\put(1.7,6){\line(2,3){1.9}}\put(2.4,5.3){\vector(1,0){4.2}}
	\put(3.7,0.3){\vector(-2,1){4.2}}\put(-.8,3){\vector(1,1){1.6}}
	\put(4.9,0.4){\line(2,3){2.8}}
	\put(3.9,0.4){\line(-2,3){2.8}}
	\put(-.6,4){\footnotesize{$\la$}}
	\put(4.5,3.2){\footnotesize{$\la'$}}
	\put(4,5.5){\footnotesize{$\tau$}}
	\put(-.5,2.7){\vector(4,1){7.5}}
	\end{picture}
\end{center}\bs

Our assumption $\vb\circ \la=\vb\circ \la'$ implies that the restriction to $\la(L)$ of the two valuations $\vb\circ\tau$ and $\vb$ coincide.
By the Conjugation theorem \cite[Theorem 14.2]{endler}, there exists $\rho\in\aut(\kb/\la(L))$ such that $\vb\circ\tau=\vb\circ \rho$.

Then, $\sg:=\tau\circ \rho^{-1}$ belongs to $\aut(\kb/\kh)$, because it satisfies $\vb\circ\sg=\vb$. Finally, since $\rho$ is the identity on $\la(L)$, we have $\rho^{-1}\circ\la=\la$. Thus,
$$
\sg\circ \la=\tau\circ \rho^{-1}\circ\la=\tau\circ\la=\la'.
$$
This ends the proof of the theorem.
\end{proof}

%For instance, $\ee(\kb)$ is in bijection with  $\aut(\kb/\kh)\backslash\aut(\kb/K)$. Every right coset  $\aut(\kb/\kh)\,\sg$ determines the valuation $\vb\circ\sg$.

\subsection{Extensions to simple finite extensions}\label{independencehens}
Suppose now $L=\kx/(g)$ for some $g\in\irr(K)$. 
Let us recall how to describe the finite number of extensions of $v$ to $L$.

Since $\kh/K$ is a separable extension, the factorization of $g$ into a product of monic irreducible polynomials in $\kh[x]$ takes the form
$$
g=G_1\cdots G_r,
$$
with pairwise distinct $G_1,\dots,G_r\in\irr(\kh)$.
 Let $Z(g)\subset \kb$ be the set of zeros of $g$, avoiding multiplicities. We have a natural bijection
$$
Z(g)\lra\mon,\qquad \t\longmapsto \la_\t, 
$$
where $\la_\t$ is determined by $\la_\t\left(x+g\kx\right)=\t$.
Clearly,
$$
\begin{array}{rccl}
\la_\t \sim_{\kh} \la_{\t'}&\sii&\la_{\t'}=\sg\circ \la_\t&\quad\mbox{for some}\quad\sg\in\aut(\kb/\kh)\\
&\sii&\t'=\sg(\t)&\quad\mbox{for some}\quad\sg\in\aut(\kb/\kh)
\end{array}
$$
Therefore, $\la_\t \sim_{\kh} \la_{\t'}$ if and only if $\t$ and $\t'$ are roots of the same irreducible factor of $g$ over $\kh[x]$.

Let us choose an arbitrary root $\t_i\in Z(G_i)$ for each irreducible factor of $g$. By Theorem \ref{endler}, the valuations $\wb_{G_i}=\vb\circ\la_{\t_i}$ do not depend on the chosen roots and are all extensions of $v$ to $L$.

\begin{corollary}\label{CorEndler}
	There are $r$ many extensions  of $v$ to $L=\kx/(g)$, given by $\wb_{G_1},\dots,\wb_{G_r}$. 
\end{corollary}

\subsection{Valuations with nontrivial support on $\khx$}

For every $G\in\irr(\kh)$, we may consider the following valuation on $\khx$:
$$
v_G\colon \khx\lra \gq\infty,\qquad g\longmapsto v_G(f)=\vb(f(\t)),
$$
where $\t\in Z(G)$.
By the henselian property, this valuation is independent on the choice of $\t$.
Clearly, $\supp(v_G)=G\khx$.

We denote the restriction of $v_G$ to $\kx$ by:
$$
w_G:=\left(v_G\right)_{\mid \kx}.
$$
Now, $\supp(w_G)=N_{\kh/K}(G)\kx$, where  $N_{\kh/K}(G)\in\irr(K)$ is the monic generator of the prime ideal  $\left(G\khx\right)\cap\kx$.
This justifies the notation of Corollary \ref{CorEndler}.

\begin{proposition}\label{CorEndler2}
For all $G\in\irr(\kh)$, this valuation $v_G$ is the unique extension of $v$ to $\khx$, with support $G\khx$. 	
Moreover, for all $F,G\in\irr(\kh)$ we have
\[
F=G\ \sii\ v_F=v_G\ \sii\ w_F=w_G.
\] 
\end{proposition}
	
\begin{proof}
There is a natural identification between valuations with a given support $G\khx$ and valuations on the field $L=\khx/(G)$.
Since $\vh$ admits a unique extension to $L$, it must be $\vb_G$, the valuation induced by $v_G$.

This proves the equivalence $F=G\ \sii\ v_F=v_G$. Finally, 
suppose that $w_F=w_G$ has support $g\kx$ for some $g\in\irr(K)$. By Corollary  \ref{CorEndler}, $F,G$ are irreducible factors of $g$ over $\khx$ and $\wb_F=\wb_G$ implies $F=G$.
\end{proof}	%\e

\subsection{Comparison of different henselizations}%\mbox{\null}\e

Any other extension of $v$ to $\kb$ is of the form
$\,\vt=\vb\circ \tau$, for some $\tau\in \aut(\kb/K)$.
The corresponding decomposition groups are conjugate:
$$
D_{\vt}=\tau^{-1}D_{\vb}\,\tau.
$$
Hence, the corresponding henselization $(\kth,v_\tau^h)$ is determined as 
$$\kth=\tau^{-1}\left(\kh\right),\qquad v_\tau^h=\vh\circ\tau.$$

\begin{center}
	\setlength{\unitlength}{4mm}
	\begin{picture}(8,5)
	\put(0.3,0){$\kh$}\put(4,4){$\ks$}
	\put(7,0){$\kth$}
	\put(7,1){\line(-2,3){1.9}}\put(1.7,1){\line(2,3){1.9}}\put(2.4,0.3){\vector(1,0){4.2}}
	\put(4,0.5){\footnotesize{$\tau^{-1}$}}
	\end{picture}
\end{center}

%\subsection{Extensions of $v$ to simple finite extensions of $K$}\mbox{\null}\e

Recall the factorization  $g=G_1\dots G_r$ of some $g\in\irr(K)$ into a product of monic irreducible factors in $\khx$.
Then, $g=\tau^{-1}(G_1)\cdots \tau^{-1}(G_r)$ is the corresponding factorization of $g$ into a product of monic irreducible factors in $\kthx$. 

\begin{lemma}\label{sameWq}
	Let $G\in\irr(\kh)$ be an irreducible factor of $g$ in $\khx$. Let $v_G$ be the unique extension of $\vh$ to $\khx$ with support $G\khx$. Let $v_{\tau^{-1}(G)}$ be the unique extension of $v_\tau^h$ to $\kthx$ with support $\tau^{-1}(G)\kthx$. 
	
	Then, their restrictions to $\kx$ coincide: \ $\left(v_G\right)_{\mid\kx}=\left(v_{\tau^{-1}(G)}\right)_{\mid\kx}$. 
\end{lemma}

\begin{proof}
	Choose any root $\t$ of $G$ in $\kb$. Then, $\tau^{-1}(\t)$ is a root of $\tau^{-1}(G)$. Clearly, 
	$$
\vt(f(\tau^{-1}(\t)))=\vb\left(\tau(f(\tau^{-1}(\t)))\right)=\vb(f(\t)),
	$$
	for all $f\in\kx$.
\end{proof}\e

\begin{corollary}\label{sameW}
	Let $G\in\irr(\kh)$ be an irreducible factor of $g$ in $\khx$.	Let $\tau^{-1}(G)\in\irr(\kth)$ be the corresponding irreducible factor of $g$ in $\kthx$. Then, the extensions of $v$ to $L=\kx/(g)$ determined by $G$ and $\tau^{-1}(G)$ coincide. 
\end{corollary}

\begin{proof}
	Denote by $w$, $w_\tau$ the extensions of $v$ to $L$ 
	determined by $G$, $\tau^{-1}(G)$, respectively. We have
	$$
	w(f+g\kx)=v_G(f),\qquad w_\tau(f+g\kx)=v_{\tau^{-1}(G)}(f),
	$$
	for all $f\in\kx$.	Thus, the corollary follows from Lemma \ref{sameWq}.
\end{proof}\e

As a consequence, we may define the \textbf{defect} of $w/v$ as
$$
d(w/v)=\dfrac{\deg(G)}{e(w/v)f(w/v)}.
$$
This is well defined because $\deg(G)$ is independent on the chosen henselization.

From this definition we derive immediately that 
\[
\sum_{i=1}^re(w_i/v)f(w_i/v)d(w_i/v)=[L\colon K],
\]
if $w_1,\dots,w_r$ are all extensions of $v$ to $L$.

Nevertheless, these arguments are still not sufficient to show that $d(w/v)$ is a positive integer.

%the extensions of $v$ to $\kx$ having support $g\kx$; that is, the valuations $w_{G_1},\dots,w_{G_r}$, are independent of the choice of $\vb$.

\section{Irreducible polynomials over henselian fields}\label{newtonpoly}
\subsection{Newton polygons}
Let $\mu$ be a valuation-transcendental valuation on $\kx$ and let $\La$ be the divisible closure of $\gm$. The choice of a key polynomial $\phi\in\kpm$ yields a \textbf{Newton polygon operator}
$$
\np\colon\, \kx\lra \pset\left({\Q\times \La}\right),
$$
where $\pset\left({\Q\times \La}\right)$ is the power set of the rational vector space $\Q\times \La$. 

The \textbf{Newton polygon} of the zero polynomial is the empty set. 
For a nonzero $g\in \kx$ with $\phi$-expansion $g=\sum\nolimits_{0\le n}a_n\phi^n$,
we define $N:=\np(g)$ as the lower convex hull in  $\Q\times \La$ of the finite set $\left\{\left(n,\mu\left(a_n\right)\right)\mid a_n\ne0\right\}$.

Thus, $N$ is either a single point or a chain of segments, $S_1,\dots, S_r$, called the \textbf{sides} of the polygon, ordered from left to right by increasing slopes.

The left and right endpoints of $N$, together with the points joining two different sides are called \textbf{vertices} of $N$.
In Figure \ref{figNmodel} we see the typical shape of such a polygon.

The abscissa of the left endpoint of $\np(g)$ is $\ord_\phi(g)$ in $\kx$. The abscissa $$\ell(\np(g))=\left\lfloor \deg(g)/\deg(\phi)\right\rfloor$$ of the right endpoint of $N$ is called the \textbf{length} of $N$. Clearly,
\begin{equation}\label{lengh}
\ell(\np(gh))\ge\ell(\np(g))+\ell(\np(h))\quad\mbox{ for all }g,h\in\kx.
\end{equation}
If $\deg(\phi)>1$, then  it is easy to construct examples where this inequality (\ref{lengh}) is strict. If $a,b\in\kx$ satisfy $\deg(a),\deg(b)<\deg(\phi)$, but $\deg(ab)>\deg(\phi)$, then
$$
\left(a\phi^m\right)\left(b\phi^n\right)=c\phi^{m+n}+q\phi^{m+n+1},
$$
where $ab=q\phi+c$ is the division with remainder of $ab$ by $\phi$.\e

%We define the \emph{length} $\ell(S_i)$ of a side as the length of its projection to the $x$-axis.\e

%\nn{\bf Convention. }{\it If $\ord_\phi(g)>0$, then we agree that $\np(g)$ contains an initial side of slope $-\infty$, whose projection to the $x$-axis is the interval $[0,\ord_\phi(g)]$.}   \e

%Thanks to this convention, $\ell(N)=\ell(S_1)+\cdots+\ell(S_r)$, provided that we count ``all" sides of $N$, including that of slope $-\infty$.\e

\begin{figure}%[h]
	\caption{Newton polygon $N=\np(g)$ of $g\in \kx$. }\label{figNmodel}
	\begin{center}
		\setlength{\unitlength}{4mm}
		\begin{picture}(20,10)
		\put(14.8,6){$\bullet$}\put(13.5,7){$\bullet$}\put(12.8,2){$\bullet$}\put(10,2){$\bullet$}\put(12,0.4){$\bullet$}\put(7.4,1.5){$\bullet$}
		\put(6.2,8.35){$\bullet$}\put(3.75,6.9){$\bullet$}\put(2.8,8.35){$\bullet$}
		\put(-1,3.6){\line(1,0){20}}\put(0,0){\line(0,1){10}}
		\put(7.6,1.8){\line(-2,3){4.5}}\put(7.6,1.83){\line(-2,3){4.5}}
		\put(7.6,1.8){\line(4,-1){4.7}}\put(7.6,1.83){\line(4,-1){4.7}}
		\put(12.2,0.5){\line(1,2){2.8}}\put(12.2,0.53){\line(1,2){2.8}}
		%\put(15,3.6){\line(1,3){1}}\put(15,3.62){\line(1,3){1}}
		\multiput(3,3.5)(0,.25){21}{\vrule height2pt}
		%\multiput(8,.9)(0,.25){9}{\vrule height2pt}
		\multiput(15,3.5)(0,.25){11}{\vrule height2pt}
		%\put(7.2,0){\begin{footnotesize}$s_{\mu,\phi}(g)$\end{footnotesize}}
		\put(1.7,2.6){\begin{footnotesize}$\ord_{\phi}(g)$\end{footnotesize}}
		\put(14.5,2.6){\begin{footnotesize}$\ell(N)$\end{footnotesize}}
		\put(18.5,2.6){\begin{footnotesize}$\Q$\end{footnotesize}}
		\put(-1,9.2){\begin{footnotesize}$ \La$\end{footnotesize}}
		%\multiput(-.1,3)(.25,0){55}{\hbox to 2pt{\hrulefill }}
		%\put(6,7){\begin{footnotesize}\end{footnotesize}}
		%\put(-1.9,2.85){\begin{footnotesize}$\mu(g)$\end{footnotesize}}
		\put(-.6,2.8){\begin{footnotesize}$0$\end{footnotesize}}
		\end{picture}
	\end{center}
\end{figure}
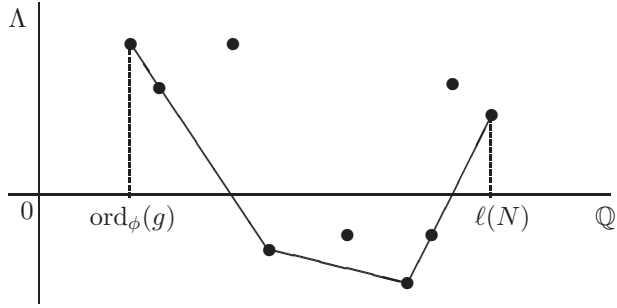\e

\defn
For all $\la\in \La$, the \mbox{\textbf{$\la$-component}} $S_\la(N)\subset N$ is the intersection of $N$ with the line $L$ of slope $-\la$ which first touches $N$ from below. In other words,
$$S_\la(N)= \{(n,\al)\in N\,\mid\, \al+n\la\mbox{ is minimal}\,\}.$$

Note that $L$ cuts the vertical axis at the point with ordinate the common value of $\al+n\la$, for all $(n,\al)\in S_\la(N)$.

The abscissas of the endpoints of $S_\la(N)$ are denoted \ $n_{\la}\le n'_{\la}$.

If $N$ has a side $S$ of slope $-\la$, then $S_\la(N)=S$. Otherwise, $S_\la(N)$ is a vertex of $N$.  Figure \ref{figComponent0} illustrates both possibilities.\e

\defn
$N$ is \textbf{one-sided} of slope $-\la$, if 
$N=S_\la(N)$, $n_{\la}=0$ and $n'_{\la}>0$.

\begin{figure}%[h]
	\caption{The line $L$ has slope $-\la$ and $S_\la(N)=L\cap N$ is the $\la$-component of $N=\np(g)$ for some $g\in\kx$. In the case $\la=\mu(\phi)$, the line $L$ cuts the vertical axis at the point with ordinate $\mu(g)$.}\label{figComponent0}
	\begin{center}
		\setlength{\unitlength}{4mm}
		\begin{picture}(30,9.6)
		\put(2.8,4.8){$\bullet$}\put(7.8,2.25){$\bullet$}
		\put(-0.2,6.5){\line(1,0){0.5}}
		\put(-1,0.6){\line(1,0){13}}\put(0,-0.4){\line(0,1){9}}
		\put(-1,7){\line(2,-1){11}}
		\put(3,5){\line(-1,2){1.5}}\put(3,5.04){\line(-1,2){1.5}}
		\put(3,5){\line(2,-1){5}}\put(3,5.04){\line(2,-1){5}}
		\put(8,2.5){\line(5,-1){4}}\put(8,2.54){\line(5,-1){4}}
		\multiput(3,.4)(0,.25){18}{\vrule height2pt}
		\multiput(8,.4)(0,.25){8}{\vrule height2pt}
		\put(7.8,-.4){\begin{footnotesize}$n'_{\la}$\end{footnotesize}}
		\put(2.6,-.4){\begin{footnotesize}$n_{\la}$\end{footnotesize}}
		\put(-1.2,7.3){\begin{footnotesize}$L$\end{footnotesize}}
		\put(-.6,-.2){\begin{footnotesize}$0$\end{footnotesize}}
		\put(-2.4,6){\begin{footnotesize}$\mu(g)$\end{footnotesize}}
		\put(5.5,4){\begin{footnotesize}$S_\la(N)$\end{footnotesize}}
		\put(20.8,5.35){$\bullet$}\put(23.8,3.3){$\bullet$}
		\put(17.75,5.58){\line(1,0){0.5}}
		\put(17,0.6){\line(1,0){13}}\put(18,-0.4){\line(0,1){9}}
		\put(17,5.9){\line(3,-1){10}}
		\put(24,3.6){\line(-3,2){3}}\put(24,3.64){\line(-3,2){3}}
		\put(21,5.8){\line(-1,2){1}}\put(21,5.84){\line(-1,2){1}}
		\put(24,3.6){\line(1,0){3}}\put(24,3.63){\line(1,0){3}}
		\multiput(24,.5)(0,.25){12}{\vrule height2pt}
		\put(22.5,-.4){\begin{footnotesize}$n_{\la}=n'_{\la}$\end{footnotesize}}
		\put(16.8,6.25){\begin{footnotesize}$L$\end{footnotesize}}
		\put(17.4,-.2){\begin{footnotesize}$0$\end{footnotesize}}
		\put(23.4,4.2){\begin{footnotesize}$S_\la(N)$\end{footnotesize}}
		\put(16,5){\begin{footnotesize}$\mu(g)$\end{footnotesize}}
		\end{picture}
	\end{center}
\end{figure}
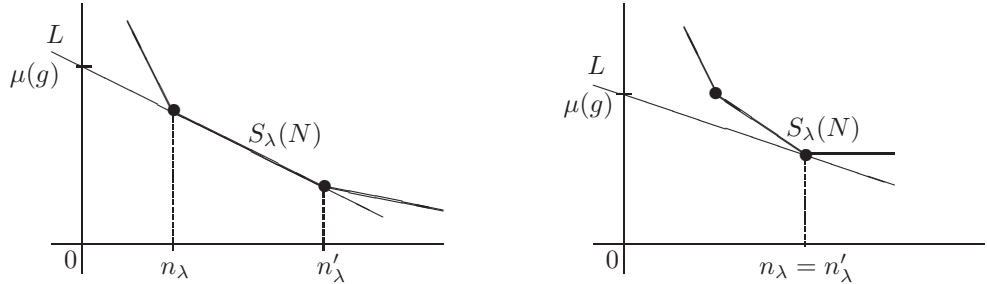\e

%\subsection{Dissection by Newton polygons}\label{subsecDissNP}\mbox{\null}\e

\defn
The \textbf{principal Newton polygon} $\npp(g)$ is the polygon formed by the sides of $\np(g)$ of slope less than $-\mu(\phi)$. 

If $\np(g)$ has no sides of slope less than $-\mu(\phi)$, then $\np^+(g)$ is defined to be the left endpoint of $\np(g)$. 
\e

Clearly, the set $S_{\mu,\phi}(g):=\left\{0\le n\mid \mu\left(a_n \phi^n\right)=\mu(g)\right\}$ coincides with the set of abscissas of the points lying \emph{on} the segment $S_{\mu(\phi)}(g)$. In particular, if  $\pi=\inm\phi$,  then
\begin{equation}\label{length}
\ell\left(\np^+(g)\right)=\min(S_{\mu,\phi}(g))=\ord_\pi(\inm g).
\end{equation}

%\subsection*{Addition of Newton polygons}\label{secAdd}

There is an addition law for Newton polygons. 
Consider two  polygons $N$, $N'$ with sides $S_1,\dots,S_r$, $S'_1,\dots,S'_{s}$, respectively. 

The left endpoint of the sum $N+N'$ is  the vector sum in $\Q\times\La$ of the left endpoints of $N$ and $N'$, whereas the sides of $N+N'$ are obtained by joining to this endpoint all sides in the multiset $\left\{S_1,\dots,S_r,S'_1,\dots,S'_{s}\right\}$, ordered by increasing slopes.

If one of the polygons is a one-point polygon (say  $N'=\{P\}$), then it has an empty set of sides and the sum $N+N'$ coincides with the vector sum $N+P$ in $\Q\times  \La$.

Note that, by construction, $\ell(N+N')=\ell(N)+\ell(N')$.
 
\begin{theorem}\label{product}
	For a  valuation-transcendental $\mu$ and any $\phi\in\kpm$, we have   
$$%	\begin{equation}\label{sumNP}
	\npp(gh)=\npp(g)+\npp(h),
$$%	\end{equation}
	for all nonzero $g,h\in \kx$.
\end{theorem}

This result follows immediately from the following observation.

\begin{lemma}\label{productSL}
Take $\phi\in\kpm$ and  $\la>\mu(\phi)$. For  all nonzero $g,h\in\kx$, we have   
\begin{equation}\label{Sl}
S_\la(gh)=S_\la(g)+S_\la(h).
\end{equation}
\end{lemma}

\begin{proof}
For any such $\la$,
consider the augmented valuation 
$\mu_\la=[\mu;\,\phi,\la]$.
As mentioned in Section \ref{subsecKP}, $\phi$ is a key polynomial of minimal degree for $\mu_\la$. Denote $\pi=\inn_{\mu_\la}\phi$.

For all $f\in\kx$, consider the $\phi$-expansion $f=\sum_{n\ge0}a_n\phi^n$, and denote $$S(f):=S_{\mu_\la,\phi}(f)=\{n\mid \mu_\la(a_n\phi^n)=\mu_\la(f)\}.$$ The expression of $\inn_{\mu_\la}f$ as a polynomial in $\pi$ with coefficients in $\gg_{\mu_\la}^0$ is
$$
\inn_{\mu_\la}f=\sum_{\ell\in S(f)}(\inn_{\mul}a_\ell)\pi^\ell.
$$
Clearly,
the indices in the set $S(f)$ correspond to abscissas of points lying on the segment (or point) $S_\la\left(\np(f)\right)$. Hence, the initial point $P(f)=\left(s,\al\right)$ and last point $Q(f)=\left(t,\be\right)$ of $S_\la(f)$ are determined by
$$
s=\min(S(f)), \quad \al=\gr_{\mu_\la}(\ab_{s}),\qquad   
t=\max(S(f)), \quad \be=\gr_{\mu_\la}(\ab_{t}).
$$
Now, the equality (\ref{Sl}) follows from 
$$P(gh)=P(g)+P(h),\qquad Q(gh)=Q(g)+Q(h),\quad \mbox{ for all }g,h\in\kx.
$$
These identities follow immediately from Theorem \ref{g0gm}.
\end{proof}\e

%The analogous statement for entire Newton polygons is false. 

\subsection{A generalization of Hensel's lemma}\label{henseliafields}
We assume in this section that the valued field $(K,v)$ is henselian. We still denote by $v$ the unique extension of $v$ to $\kb$.

The following result is well-known.

\begin {lemma}\label{symmetry}
For all $F,G\in\irr(K)$, we have \ $v_G(F)/\deg(F)=v_F(G)/\deg(G)$. 
\end {lemma}

The following theorem is the main result of this section. 

\begin{theorem}\label{fundamental}
	Let $Q\in\kpm$ for some valuation $\mu$ on $\kx$. Then,
	\begin{equation}\label{aim}
\inv_\mu Q\mid \inv_\mu F\ \sii\ \mu<v_F \ \mbox{ and }\ \ty(\mu,v_F)=[Q]_\mu,
	\end{equation}
	for all $F\in\irr(K)$.
	Moreover, if these conditions hold, then:
	\begin{enumerate}
		\item[(i)] Either $F=Q$, or the Newton polygon $N_{\mu,Q}(F)$ is one-sided of slope $-v_F(Q)$.    
		\item[(ii)] $F\smu Q^\ell$ with $\ell=\ell(N_{\mu,Q}(F))=\deg(F)/\deg(Q)$. 
	\end{enumerate}
\end{theorem}

\begin{proof}
	If $F=Q$, then  all statements of the theorem are trivial. Assume $F\ne Q$.
	
	If $\mu<v_F$  and $\ty(\mu,v_F)=[Q]_\mu$, then $\inv_\mu Q\mid \inv_\mu F$ by Lemma \ref{tdef}, because $\mu(F)<v_F(F)=\infty$.
	
	Conversely, suppose $\inv_\mu Q\mid \inv_\mu F$. Since  $\inv_\mu F$ is not a unit in $\ggm$, we have $\mu<v_F$ by  \cite[Theorems 1.1,1.3]{BN}. By Lemma \ref{td}, in order to show that $\ty(\mu,v_F)=[Q]_\mu$ it suffices to check that $\mu(Q)<v_F(Q)$.

	Choose a root $\t\in\kb$  of $F$, and consider the minimal polynomial of $Q(\t)$ over $K$: 
	$$g=b_0+b_1x+\cdots+b_kx^k\in K[x],\qquad b_k=1.$$ 
	All roots of $g$ in $\kb$ have a constant $v$-value $\ep:=v(Q(\t))=v_F(Q)$. Hence,
	\begin{equation}\label{slope}
	v(b_0)=k\ep,\quad v(b_j)\ge (k-j)\ep, \ 1\le j<k,\quad v(b_k)=0.
	\end{equation}
	Let $G=\sum_{j=0}^kb_jQ^j\in\kx$. By (\ref{slope}), $N_{\mu,Q}(G)$ is one-sided of slope $-\ep$. 	
	Since $G(\t)=0$, the polynomial $F$ divides $G$ in $\kx$. If $G=FH$, then Theorem \ref{product} shows that
	\begin{equation}\label{sum}
	N^+_{\mu,Q}(G)=N^+_{\mu,Q}(F)+N^+_{\mu,Q}(H).
	\end{equation}
	By equation (\ref{length}), $\ell(N^+_{\mu,Q}(F))>0$, because $\inv_\mu Q\mid \inv_\mu F$. Then, (\ref{sum}) shows that  $N^+_{\mu,Q}(G)$ has positive length too. This implies
	$N_{\mu,Q}(G)=N^+_{\mu,Q}(G)$ and $\mu(Q)<\ep$, by the definition of the principal polygon. This ends the proof of (\ref{aim}).

	On the other hand, since $N^+_{\mu,Q}(G)$ is one-sided of slope $-\ep$, (\ref{sum}) shows that $N^+_{\mu,Q}(F)$ is one-sided of slope $-\ep$ too. 
	Finally, from the inequality, 
	\begin{align*}
		\ell(N_{\mu,Q}(G))&\;\ge\;\ell(N_{\mu,Q}(F))+\ell(N_{\mu,Q}(H))\ge\ell(N^+_{\mu,Q}(F))+\ell(N^+_{\mu,Q}(H))\\&\;=\;\ell(N^+_{\mu,Q}(G))=\ell(N_{\mu,Q}(G)).
	\end{align*}	
	we deduce that $\ell(N_{\mu,Q}(F))=\ell(N^+_{\mu,Q}(F))$. Therefore, $N_{\mu,Q}(F)=N^+_{\mu,Q}(F)$ is one-sided of slope $-\ep$. 	
	This proves that $\inv_\mu Q\mid \inv_\mu  F$ implies that the property (i) holds.

	\begin{figure}
		\caption{Newton polygon $N_{\mu,Q}(F)$ when $\inv_\mu Q\mid \inv_\mu F$.  We take $\ep=v_F(Q)$}\label{figNF}
		\begin{center}
			\setlength{\unitlength}{4mm}
			\begin{picture}(10,7)
			\put(9.8,1.25){$\bullet$}\put(-.25,4.75){$\bullet$}
			\put(-1,0.6){\line(1,0){14}}\put(0,-.4){\line(0,1){7}}
			\put(0,5){\line(3,-1){10}}
			\multiput(10,.5)(0,.25){5}{\vrule height2pt}
			\multiput(-.1,1.6)(.25,0){40}{\hbox to 2pt{\hrulefill }}
			\put(-2.6,4.8){\begin{footnotesize}$\mu(a_0)$\end{footnotesize}}
			\put(-2.5,1.5){\begin{footnotesize}$\mu(a_\ell)$\end{footnotesize}}
			\put(-.6,-.4){\begin{footnotesize}$0$\end{footnotesize}}
			\put(9.8,-.5){\begin{footnotesize}$\ell$\end{footnotesize}}
			\put(4.5,4){\begin{footnotesize}slope $-\ep$\end{footnotesize}}
			\end{picture}
		\end{center}
	\end{figure}
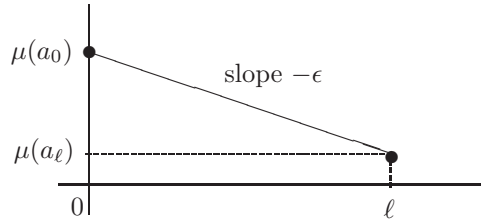
	
	In order to prove (ii), consider the $Q$-expansion $F=\sum\nolimits_{n=0}^\ell a_nQ^n$.	By Lemma \ref{symmetry}, $$\mu(a_0)=v_Q(F)=\deg(F)v_F(Q)/\deg(Q)=\deg(F)\ep/\deg(Q).$$ 
	Therefore, since $N_{\mu,Q}(F)$ is one-sided of slope $-\ep$, a look at Figure \ref{figNF} shows that: 
	$$%\begin{align*}
	\mu(a_\ell)=\mu(a_0)-\ell\ep=\ep\,\dfrac{\deg(F)}{\deg(Q)}-\ell\ep=\ep\,\dfrac{\deg(a_\ell)+\ell\deg(Q)}{\deg(Q)}-\ell\ep=\ep\,\dfrac{\deg(a_\ell)}{\deg(Q)}.
	$$%\end{align*}
	If $\deg(a_\ell)>0$, then $a_\ell$ would be a monic polynomial contradicting \cite[Theorem 3.9]{KP}:
	$$
	\mu(a_\ell)/\deg(a_\ell)=\ep/\deg(Q)>
	\mu(Q)/\deg(Q).
	$$ 
	Hence, $a_\ell=1$, so that the leading monomial of the $Q$-expansion of $F$ is $Q^\ell$. In particular, $\ell=\deg(F)/\deg(Q)$.
	
	Since $\mu(Q)<\ep$, we have $\mu(Q^\ell)<\mu(a_nQ^n)$ for all $n<\ell$. Thus, $F\smu Q^\ell$.
\end{proof}\e

\begin{corollary}\label{unit/minimal}
	Let $\mu$ be a valuation on $\kx$ and $F\in\irr(K)\setminus \supp(\mu)$. Then, $F$ is $\mu$-minimal if and only if $\inv_\mu F$ is not a unit in $\mathcal G_\mu$.
\end{corollary}

\begin{proof}
	If $F$ is $\mu$-minimal, then $F\nmid_\mu 1$, so that $\inv_\mu F$ is not a unit in $\ggm$.
	
	Conversely,	if $\inv_\mu F$ is not a unit, then  \cite[Theorem 6.8]{KP} shows that there exists $Q\in\kpm$ such that $\inv_\mu Q\mid \inv_\mu F$. By Theorem \ref{fundamental}, $\deg(F)=\ell\deg(Q)$ and $\inv_\mu F=\inv_\mu Q^\ell$  for some $\ell\in\Z_{>0}$.
	By Lemma \ref{lemma1x}, $F$ is $\mu$-minimal.
\end{proof}\e

\section{Strong rigidity of valuations under henselization}\label{prooofthm1}

Let  $\g\hk \La$ be a fixed order-preserving embedding of $\g$ into a divisible ordered abelian group $\La$.

Let $\ttt$, $\ttt^h$ be the trees of $\La$-valued extensions of $v$, $\vh$  to $\kx$, $\khx$, respectively.

The following rigidity statement was proven in \cite[Theorem A]{Rig}.
 
\begin{theorem} \label{Rig}
The restriction of valuations  induces an isomorphism of posets:
	$$
	\ttt^h\lra\ttt,\qquad \eta\ \longmapsto \ \eta_{\,\mid \kx}.
	$$	
\end{theorem}

In particular, for any valuation $\mu\in\ttt$, there is a unique common extension $\muh$ of $\mu$ and $\vh$ to  $\khx$. 

The aim of this section is to prove Theorem \ref{theoremabouthensel}. That is, the canonical embedding $\ggm\hk\ggmh$ is an isomorphism of graded algebras.

This strong-rigidity statement will be used to analyze the extensions of augmentations $\mu\to\nu$ from $\kx$ to $\khx$.

\subsection{A general rigidity criterion}
Let $\mu$ be an arbitrary valuation on $\kx$.
Let $L/K$ be an algebraic field extension and $\eta$ a valuation on $L[x]$ extending $\mu$. 

The mapping $\inm g\mapsto \inv_\eta g$, for all $g\in\kx\setminus \supp(\mu)$, induces a grade-preserving embedding between the respective graded algebras:
\begin{equation}\label{iota}
\iota\colon \ggm\hk\mathcal G_\eta. 
\end{equation}

We say that $\iota$ \textbf{preserves the degree} if 
$\deg_\eta(f)=\dgm(f)$,
for all $f\in\kx$.

We write $e(\eta/\mu)=1$ to indicate that the embedding $\gm\hk\Gamma_\eta$ is onto. 

We write $f(\eta/\mu)=1$ to indicate that the field extension $\km\hk k_\eta$ is onto. 

Let $\dm\subset\ggm$ be the subring of homogeneous elements of grade zero of $\ggm$, and let $\dm^*\subset\dm$ be the multiplicative subgroup of all units in $\dm$. Let  $\Delta^*_\eta\subset\Delta_\eta\subset\gg_\eta$ be the analogous objects in $\gg_\eta$.

\begin{lemma}\label{previous}
Suppose that $\gm^0=\geta^0$ and the embedding $\iota$ in (\ref{iota}) induces an isomorphism between $\dm^*$ and $\Delta^*_\eta$. Then, $\iota$ induces an isomorphism between $\ggm^0$ and $\gg^0_\eta$. 
\end{lemma}

\begin{proof}
Since $\iota$ maps units to units, it induces an embedding $\ggm^0\hk\gg^0_\eta$. We must show that this embedding is onto.  It suffices to show that every homogeneous unit in $\gg_\eta$ is the image of an homogeneous unit in $\ggm$.

Take any $F\in L[x]\setminus\supp(\eta)$ such that  $\inn_\eta F$ is a unit. Since $\eta(F)$ belongs to $\geta^0=\ggm^0$, there exists $f\in\kx$ such that $\eta(F)=\mu(f)=\eta(f)$ and $\inm f$ is a unit.

Hence,  $(\inv_\eta F)(\inv_\eta f)^{-1}\in\Delta_\eta^*$. By our assumption,  there exists $h\in\kx$ such that $\inm h\in \dm^*$ and $\inv_\eta h=(\inv_\eta F)(\inv_\eta f)^{-1}$. Hence, $\inm(gh)$ is mapped to $\inv_\eta F$ by $\iota$. 
\end{proof}\e

\begin{lemma}\label{crit}
Let $L/K$ be an algebraic field extension and $\eta$ a valuation on $L[x]$. Let $\mu=\eta_{\mid \kx}$ and $\iota\colon \ggm\hk\mathcal G_\eta$ the canonical embedding of graded algebras. 
	Then, the following conditions are equivalent.
	\begin{enumerate}
		\item [(i)] The embedding $\iota$ preserves the degree and induces an  isomorphism $\ggm^0\simeq\mathcal G_\eta^0$.
		\item [(ii)] The embedding $\iota$ is an isomorphism.
		\item [(iii)] $e(\eta/\mu)=1=f(\eta/\mu)$.
	\end{enumerate}	
\end{lemma}

\begin{proof}
	(i) implies (ii)

	If $\mu$ is valuation-algebraic or has nontrivial support, then  $\ggm^0=\ggm$. On the other hand, since $L/K$ is algebraic, $\eta$ must be of the same type (valuation-algebraic or nontrivial support) too. Thus, (ii) follows trivially from (i).
	
	If $\mu$ is valuation-transcendental, then (ii) follows from (i) by Theorem \ref{g0gm}. Indeed, let $\ggm=\ggm^0[\pi]$ for some 
	homogeneous $\pi\in\ggm$ with $\dgm(\pi)=1$. From (i) we deduce $\iota\left(\ggm\right)=\gg_\eta^0[\iota(\pi)]$ with $\deg_\eta(\iota(\pi))=1$. Hence, $\gg_\eta^0[\iota(\pi)]=\mathcal G_\eta$, again by Theorem  \ref{g0gm}.\e 
	
	(ii) implies (iii)
	
	For any $F\in L[x]$, there is some $f\in\kx$ such that $\inv_\eta f=\inv_\eta F$. Hence, $\eta(F)=\eta(f)=\mu(f)$ belongs to $\gm$. This proves that $e(\eta/\mu)=1$.
	
	The isomorphism $\iota$ restricts to an isomorphism $\dm\simeq \Delta_\eta$. By \cite[Section 4]{KP}, the residue fields $\km$, $k_\eta$ are the fields of fractions of  $\dm$, $\Delta_\eta$, respectively. Hence, the natural embedding $\km\hk k_\eta$ is onto, so that $f(\eta/\mu)=1$.\e
	
	(iii) implies (i)
	
	Let us first suppose that $\mu$ is valuation-algebraic or it has nontrivial support.
	As mentioned above, $\eta$ is of the same type as $\mu$. By \cite[Theorem 4.3]{KP},
	$$\ggm^0=\ggm,\quad \dm=\km,\qquad \mathcal G_\eta^0=\mathcal G_\eta,\quad \Delta_\eta=k_\eta.$$
	Thus, $\iota$ preserves the degree, because all homogeneous elements have degree zero.
	
	Since $\gm=\gm^0$ and $\geta=\geta^0$, the condition $e(\eta/\mu)=1$ implies $\gm^0=\geta^0$. Also, the condition  $f(\eta/\mu)=1$ implies that $\iota(\dm^*)=\Delta_\eta^*$. By Lemma \ref{previous}, $\iota$ induces an isomorphism between  $\ggm^0$ and $\gg_\eta^0$.

	From now on,  assume that $\mu$ (hence $\eta$) is valuation-transcendental.

	Let $\pi=\inm\phi\,$ for some fixed $\phi\in\kpm$ of minimal degree. Since $\iota$ maps homogeneous units to homogeneous units, Theorem \ref{g0gm} shows that $\iota$ preserves the degree if and only if $\deg_\eta(\phi)=1$.

	Suppose that $\mu$ (hence $\eta$) is value-transcendental. We have again $\dm=\km$, $\Delta_\eta=k_\eta$ \cite[Theorem 4.2]{KP}. 
	By \cite[Lemma 4.1]{KP}
	we have a group isomorphism
	\[
	\gm^0\times \Z\lra \gm,\qquad (\al,\ell)\longmapsto \al+\ell\mu(\phi).
	\]
Thus, $\gm^0$ is the subgroup of all elements in $\gm$ having torsion over $\g$. Hence, the condition $e(\eta/\mu)=1$ implies   $\gm^0=\geta^0$. As in the previous case, $\iota$ induces an isomorphism between  $\ggm^0$ and $\gg_\eta^0$, by Lemma \ref{previous}. 

On the other hand, if $\deg_\eta(\phi)=m$, then Theorem \ref{g0gm} shows that
	the images by $\iota$ of all homogeneous elements in $\ggm$ take values in the subgroup $\Gamma_\eta^0\times m\Z$. Since $e(\eta/\mu)=1$, this subgroup must be the whole group $\mathcal G_\eta^0\times \Z$. Hence, $m=1$ and  $\iota$ preserves the degree.

	Finally, suppose that $\mu$ is residue-transcendental. Then, $\km=\op{Frac}\left(\dm\right)$ is transcendental over $k$. More precisely,  $\dm=\ka_\mu[\xi_\mu]$, where the field $\ka_\mu\subset\dm$ is the relative algebraic closure of $k$ in $\dm$, and $\xi_\mu\in\dm$ is transcendental over $\ka_\mu$ with $\dgm(\xi_\mu)=e(\mu)$ \cite[Theorem 4.6]{KP}. 
	
	Hence, $\eta$ is residue-transcendental too, and we have analogous statements: 
	$$ \Delta_\eta=\ka_\eta[\xi_\eta],\quad k_\eta=\op{Frac}\left(\Delta_\eta\right),
	$$
	with $\xi_\eta\in\Delta_\eta$ transcendental over $\ka_\eta$, having degree $\deg_\eta(\xi_\eta)=e(\eta)$ .
	
	The embedding $\iota$ restricts to an embedding $$\dm=\ka_\mu[\xi_\mu]\hk \dn=\ka_\nu[\xi_\nu],$$ and the condition $f(\eta/\mu)=1$ implies that this embedding induces an isomorphism between the corresponding fields of fractions. Hence, the embedding $\ka_\mu\hk\ka_\eta$ is necessarily onto  and 
	\begin{equation}\label{xi}
	\iota(\xi_\mu)=a\xi_\eta+b \quad\mbox{ for some }\quad a,b\in\ka_\eta.  
	\end{equation}
	%In particular, $\iota$ induces an isomorphism between $\dm$ and $\dn$. %\e

	Now, since the homomorphism $\iota$ maps units to units and $e(\eta/\mu)=1$, we have $\gm^0\subset\Gamma_\eta^0\subset\Gamma_\eta=\gm$; thus, $e(\mu)$ is a multiple of $e(\eta)$. 
	
	Finally, if $m=\deg_\eta(\phi)$, then  we deduce from (\ref{xi}) that
	$$
	me(\mu)=\deg_\eta(\iota(\xi_\mu))=\deg_\eta(\xi_\eta)=e(\eta).
	$$
	Hence, $m=1$ and $e(\mu)=e(\eta)$. Therefore, $\iota$ preserves the degree and $\gm^0=\geta^0$. By Lemma \ref{previous}, $\iota$ induces an isomorphism between  $\ggm^0$ and $\gg_\eta^0$.
\end{proof}\e

We may apply this criterion to prove Theorem \ref{theoremabouthensel} for valuations with trivial support.
\begin{theorem}\label{StrongRigVT}
	Let $\mu$ be a valuation on $\kx$ with trivial support. Then, the canonical embedding $\ggm\hk\ggmh$ is an isomorphism. 
\end{theorem}

\begin{proof}
	Since $\supp(\mu)=(0)$ it extends, in an obvious way, to a valuation on the field $K(x)$. Consider the following situation.
	\begin{equation}\label{sit}
	\left\{\begin{array}{ll}
	\vb & \mbox{a fixed extension of }v\mbox{ to a certain fixed algebraic closure } \kb,\\
	\kh & \mbox{the corresponding henselization,}\\
	\kxb & \mbox{a fixed algebraic closure of }K(x)\mbox{, compatible with }\kb,\\
	\omb &\mbox{an arbitrary common extension of $\mu$ and $\vb$  to }\kxb,\\
	K(x)^h &\mbox{the henselization of }(K(x),\mu)\mbox{ with respect to this choice of }\omb,\\
	\mub &\mbox{the restriction of }\omb\mbox{ to }K(x)^h.
	\end{array}\right.
	\end{equation}
	
	Since $\kh\subset K(x)^h$ we can consider the following network of valued fields:
	
	\begin{center}
		\setlength{\unitlength}{4mm}
		\begin{picture}(20,13)
		\put(0,11){$(\kb(x),\overline \omega)$}
		\put(-0.5,7){$(\kh(x),\mu^h)$}
		\put(-0.1,0.9){$(K(x),\mu)$}
		\put(13,11){$(\kxb,\overline \omega)$}
		\put(13,7){$(K(x)^h,\overline \mu)$}
		\put(2,2.2){\line(0,1){4.3}}
		\put(4.6,1.1){\line(2,1){10.7}}
		\put(4.8,11.4){\line(1,0){8}}
		\put(4.8,7.4){\line(1,0){8}}
		\put(2,8.4){\line(0,1){2}}
		\put(15.5,8.4){\line(0,1){2}}
		\end{picture}
	\end{center}
	
	By Theorem \ref{Rig}, since the restriction of $\mub$ to $K(x)$ is $\mu$, its restriction to $\kh(x)$ must be the unique extension  $\muh$ of $\mu$ to $\kh(x)$. 
	
	Since $e(\mub/\mu)=1=f(\mub/\mu)$, we deduce that 
	$$
	e(\muh/\mu)=1=f(\muh/\mu).
	$$
	Hence, the canonical embedding $\ggm\hk\gg_{\muh}$ is an isomorphism, by Lemma \ref{crit}. 
\end{proof}\e

Next, we proceed to apply this partial version of Theorem \ref{theoremabouthensel} to analyze the behavior of augmentations under henselization.

\subsection{Strong rigidity of augmentations}

%\subsection{Irreducible factors of key polynomials over a henselization}

Let $\mu$ be a valuation on $\kx$ and $\phi\in\kpm$. The valuation $[\mu;\phi,\infty]$ has support $\phi\kx$. 

Now, let $\cc$ be a continuous family of valuations on $\kx$ and  $\phi\in\kpi(\cc)$. The valuation $[\cc;\phi,\infty]$ has support $\phi\kx$.

By Proposition \ref{CorEndler2}, in both cases there exists a unique irreducible factor $Q\in\irr(\kh)$ of $\phi$ in $\khx$ such that
$$
[\mu;\phi,\infty]=w_Q=(v_Q)_{\mid \kx},\quad \mbox{ or }\quad 
[\cc;\phi,\infty]=w_Q=(v_Q)_{\mid \kx}.
$$

\begin{definition}
	We denote this irreducible factor of $\phi$  by $Q_{\mu,\phi}$, $Q_{\cc,\phi}$, respectively. It is called the {\bf irreducible factor} of $\phi$ over $\khx$ {\bf singled out} by the valuation $\mu$, or the continuous family $\cc$.
\end{definition}

\begin{proposition}\label{P3.3Isom}
	For all $\phi\in\kpm$, the irreducible factor $Q=Q_{\mu,\phi}\in\irr(\kh)$ is a key polynomial for $\muh$ and $\inmh(\phi/Q)$ is a unit in $\ggmh$. Moreover, $\dgm(\phi)=\dgmh(Q)$. 
	In particular, if $\phi$ has minimal degree in $\kpm$, then $Q$ has minimal degree in $\kp(\muh)$.
\end{proposition}

\begin{proof}
	Since $\inm \phi$ is a prime element in $\ggm$, its image $\inmh \phi$ is a prime element in $\ggmh$, by Theorem \ref{StrongRigVT}. If $\phi=Q_1\cdots Q_r$ is the factorization of $\phi$ into irreducible factors in $\khx$, then all $\inmh Q_i$ must be units, except for one factor, for which $\inmh Q_i$ is a prime.
	
	By Theorem \ref{Rig}, $\muh< v_Q$. Hence,
	$\inmh Q$ is divided by the prime class $\ty(\muh,v_Q)$, and this implies that $\inmh Q$ cannot be a unit. Therefore, $\inmh Q$ is a prime in $\ggmh$ and $\inmh (\phi/Q)$ is a unit.
	Also, $Q$ is $\muh$-minimal by Corollary \ref{unit/minimal}; hence $Q$ is a key polynomial for $\muh$.
	
	Since $\inmh (\phi/Q)$ is a unit, we have $\dgmh(\phi/Q)=0$, so that $\dgmh(\phi)=\dgmh(Q)$. 
	By Lemma \ref{crit}, the isomorphism $\ggm\hk\ggmh$ preserves the degree; thus, $\dgm(\phi)=\dgmh(\phi)=\dgmh(Q)$. In particular $\dgm(\phi)=1$ implies $\dgmh(Q)=1$.
\end{proof}

\begin{lemma}\label{Ordinary}
	For some $\phi\in\kpm$, let $Q=Q_{\mu,\phi}$. Consider the ordinary augmentation $\nu=[\mu;\,\phi,\ga]$, for some $\ga\in\La\infty$ such that $\ga>\mu(\phi)$. Then, \[\nuh=[\muh;\, Q,\ga'], \quad \mbox{ where }\ \ga'=\ga-\muh(\phi/Q).
	\]
Moreover, $Q=Q_{\nu,\phi}$ as well.	
\end{lemma}

\begin{proof}
	By Theorem \ref{Rig}, the restriction mapping from $\khx$ to $\kx$ establishes an isomorphism of totally ordered sets between the intervals $[\muh,v_Q]$ and $[\mu,w_Q]$.
	Since $\mu<\nu\le w_Q$, we deduce that $\muh<\nuh\le v_Q$. 

	By Proposition \ref{P3.3Isom}, $Q\in\kp(\muh)$ and $\inmh (\phi/Q)$ is a unit. Since, $Q\in\kp(\muh)$, we have $v_Q=[\muh;Q,\infty]$ by Proposition \ref{CorEndler2}. By \cite[Lemma 2.7]{MLV}, there exists $\ga'\in\La\infty$,  such that $\muh(Q)<\ga'\le\infty$ and $\nuh=[\muh;Q,\ga']$.
	
	Clearly, $\ga'=\infty$ if $\ga=\infty$. Otherwise, since $\nuh(\phi)=\nu(\phi)=\ga$, we deduce that $\ga'=\nuh(Q)=\ga-\nuh(\phi/Q)$. Finally, since $\inmh (\phi/Q)$ is a unit, we have $\muh(\phi/Q)=\nuh(\phi/Q)$.  
	
Finally, as in every augmentation, $\phi$ and $Q$ become  key polynomials for $\nu$ and $\nuh$, respectively. By Proposition   \ref{P3.3Isom}, $Q_{\nu,\phi}$ is a key polynomial for $\nuh$ too.

Since  $\inmh (\phi/Q)$ is a unit in $\ggmh$, necessarily $\inn_{\nuh}  (\phi/Q)$ is a unit in $\ggnh$ too. Thus, $Q$ is the unique irreducible factor of $\phi$ such that $\inn_{\nuh}Q$ is not a unit. Hence, we must have $Q=Q_{\nu,\phi}$
\end{proof}\e

Now, let us analyze the situation of limit augmentations.

\begin{definition}
Let $\cc=\left(\ri\right)_{i\in A}$ be a continuous family of valuations on $\kx$ and let $\phi\in\kpi(\cc)$ be a limit key polynomial. Consider the supremum
$$\ga_\cc=\sup\{\ri(\phi)\mid i\in A\},$$
taking values in some chosen extension $\La\hk\Hat{\La}$ containing the completion of $\La$ with respect to the order topology. 

The limit augmentation $\nu_\cc:=[\cc;\, \phi, \ga_\cc]$, taking values in $\Hat{\La}$, is called the \textbf{minimal limit augmentation} of $\cc$. 

The value $\ga_\cc$ and the valuation $\nu_\cc$ do not depend on the choice of $\phi\in\kpi(\cc)$.
\end{definition}	

In \cite{csme}, a concrete minimal model for $\Hat{\La}$ is constructed.

\begin{lemma}\cite[Section 7.3.3]{VT}\label{MinAug}
Let $\nu_\cc$ be the minimal limit augmentation of  a continuous family $\cc$. 
 Then, $\nu_\cc$ is value-transcendental and $\kp(\nu_\cc)=\kpi(\cc)$.
 Moreover, every limit augmentation $\nu=[\cc;\phi,\ga]$ can be obtained as the ordinary augmentation $\nu=[\nu_\cc;\phi,\ga]$. 	
\end{lemma}

\begin{lemma}\label{MinUnstable}
	Let $\cc$ be a continuous family of valuations on $\kx$. Let $\nu_\cc$ be the minimal limit augmentation of $\cc$. Then, for all $f\in\kx$ we have
\[%	\begin{equation}\label{MinUnst}
	\inn_{\nu_\cc} f \ \mbox{ is a unit in }\ \gg_{\nu_\cc} \ \sii\ f\ \mbox{ is $\cc$-stable}.
\]%	\end{equation}
\end{lemma}

\begin{proof}
	If $f$ is $\cc$-stable, then $\ri(f)=\rj(f)=\nu_\cc(f)$ for some $i<j$ in $A$. Hence, $\inn_{\nu_\cc} f$ is a unit in $\gg_{\nu_\cc}$, because it is the image of $\inn_{\ri}f$ under the canonical homomorphism $\gg_{\ri}\to \gg_{\nu_\cc}$ \cite[Corollary 2.6]{MLV}.

Conversely, suppose that $\inn_{\nu_\cc} f$ is a unit in $ \gg_{\nu_\cc}$. Take any limit key polynomial $\phi$ and consider the $\phi$-expansion $f=\sum_{n\ge0}f_n\phi^n$. Since $\ga_\cc$ has no torsion over $\La$, \cite[Proposition 3.5]{KP} shows that
\[
\rho_\cc(f_0)<\rho_\cc(f_n) + n\ga_\cc\quad\mbox{ for all }n>0. 
\]
Take $i_0\in A$ large enough so that $\rho_\cc(f_n)=\rho_i(f_n)$ for all $n$ and all $i\ge i_0$. By the definition of the supremum, there exists $i\ge i_0$ such that    $\rho_\cc(f_0)<\rho_\cc(f_n) + n\ri(\phi)$. This implies $\rj(f)=\ri(f)=\rho_\cc(f_0)$ for all $j\ge i$, so that $f$ is $\cc$-stable. 
\end{proof}

%\begin{lemma}\label{ggLim}	Let $\cc$ be a continuous family of valuations on $\kx$. Let $\nu$ be either the stable limit or a limit augmentation of $\cc$. Then, the set of homogeneous units of $\ggn$ coincides with $\{\inu f\mid f\in\kx\setminus\supp(\nu)  \mbox{ is $\cc$-stable}\}$.\end{lemma}

%\begin{proof}	If $f$ is $\cc$-stable, then $\ri(f)=\rj(f)=\nu(f)$ for some $i<j$ in $A$. Hence, $\inu f$ is a unit in $\ggn^*$, because it is the image of $\inn_{\ri}f$ under the canonical homomorphism $\gg_{\ri}\to\ggn$ \cite[Corollary 2.6]{MLV}.
	
%	If $\nu=\lim(\cc)$, then all polynomials in $\kx$ are stable. On the other hand, $\ggn=\ggn^0$, so that the satement of the lemma  is  obvious.  
	
%	Now, suppose that $\nu=[\cc;\phi,\ga]$ is a limit augmentation and consider any $f\in\kx\setminus\supp(\nu)$ such that $\inu f$ is a unit in $\ggn$. Since $\phi$ is a key polynomial for $\nu$ of minimal degree, we may assume that $\deg(f)<\deg(\phi)$. Hence, $f$ is $\cc$-stable.\end{proof}

\begin{lemma}\label{unstabilitythm}
	Let $\cc=\left(\ri\right)_{i\in A}$ be a continuous family of valuations on $\kx$ and $\cc^h$ its unique extension to a continuous family on $\khx$. Let $\phi$ be a limit key polynomial for $\cc$ and
consider the limit augmentation $\nu=[\cc;\phi,\ga]$, for some $\ga\in\La\infty$ such that $\ga>\ri(\phi)$ for all $i\in A$. Then, $Q:=Q_{\cc,\phi}$ is a limit key polynomial for $\cc^h$, $\phi/Q$ is $\cc^h$-stable and
\[\nuh=[\cc^h;\, Q,\ga'], \quad \mbox{ where }\ \ga'=\ga-(\rho_\cc)^h(\phi/Q).
\]	
Moreover, $Q=Q_{\nu_\cc,\phi}=Q_{\nu,\phi}$.
\end{lemma}

\begin{proof}
By Lemma \ref{MinAug}, the minimal limit augmentations $\nu_\cc$, $\nu_{\cc^h}$ are the minimal valuations which are larger than all $\ri\in\cc$, $\ri^h\in\cc^h$, respectively. By Theorem \ref{Rig}, we have $\nu_{\cc^h}=\left(\nu_\cc\right)^h$.

Let $P=Q_{\nu_\cc,\phi}$.
By Proposition \ref{P3.3Isom}, $P$ is a key polynomial for $\left(\nu_\cc\right)^h$ and $\inn_{\left(\nu_\cc\right)^h}(\phi/P)$ is a unit.
By Lemma \ref{MinUnstable}, the polynomial $\phi/P$ is $\cc^h$-stable.

By  Lemma \ref{MinAug}, $P$ is a limit key polynomial for $\cc^h$.  Also, we deduce $P=Q$ from $$w_Q=[\cc;\phi,\infty]=[\nu_\cc;\phi,\infty]=w_P.$$

By Lemma \ref{Ordinary} applied to the ordinary augmentation $\nu=[\nu_\cc;\phi,\ga]$, we have \[\nuh=[\left(\nu_\cc\right)^h; Q,\ga'], \quad \mbox{ where }\ \ga'=\ga-\left(\nu_\cc\right)^h(\phi/Q),
\]
and $Q=Q_{\nu,\phi}$. Since $\phi/Q$ is $\cc^h$-stable, we have $\left(\nu_\cc\right)^h(\phi/Q)=\rho_{\cc^h}(\phi/Q)$. Finally, $\nuh=[\cc^h; Q,\ga']$ by Lemma \ref{MinAug}. This ends the proof.
\end{proof}\e

\begin{figure}%[ht]
	\caption{Isomorphism determined by the restriction mapping}\label{figPathsLim}
	\begin{center}
		\setlength{\unitlength}{4mm}
		\begin{picture}(27,6.5)
		\put(-0.2,4){$\cc^h$}
		\put(1.7,1){\begin{footnotesize}$\cdots$\end{footnotesize}}
		\put(3,1){\begin{footnotesize}$\cdots$\end{footnotesize}}
		\put(1.7,4){\begin{footnotesize}$\cdots$\end{footnotesize}}
		\put(3,4){\begin{footnotesize}$\cdots$\end{footnotesize}}
		\put(5,5){$\nu_{\cc^h}$}\put(5,0){$\nu_\cc$}\put(5,4){$\circ$}
		\put(10,5){$\nuh$}\put(10,0){$\nu$}\put(10,4){$\bullet$}
		\put(5.4,4.3){\line(1,0){14}}\put(23,4){$\bullet$}
		\put(21.5,4){\begin{footnotesize}$\cdots$\end{footnotesize}}
		\put(20.2,4){\begin{footnotesize}$\cdots$\end{footnotesize}}
		\put(24,4){\begin{footnotesize}$v_Q=[\cc^h;Q,\infty]$\end{footnotesize}}
		\put(0,0.9){$\cc$}
		\put(5,1){$\circ$}\put(10,1){$\bullet$}\put(5.4,1.3){\line(1,0){14}}\put(23,1){$\bullet$}
		\put(21.5,1){\begin{footnotesize}$\cdots$\end{footnotesize}}
		\put(20.2,1){\begin{footnotesize}$\cdots$\end{footnotesize}}
		\put(24,1){\begin{footnotesize}$w_Q=[\cc;\phi,\infty]$\end{footnotesize}}
		\put(.3,1.9){\line(0,1){1.9}}\put(5.3,1.8){\line(0,1){2}}\put(10.2,1.8){\line(0,1){2}}
		\put(23.3,1.8){\line(0,1){2}}
		\end{picture}
	\end{center}
\end{figure}
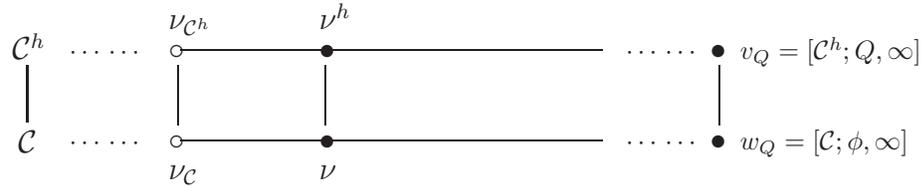

We may proceed to complete the proof of Theorem \ref{theoremabouthensel}.

\begin{theorem}\label{StrongRigEmpty}
	Let $\mu$ be a valuation on $\kx$ with nontrivial support. Then, the canonical embedding $\ggm\hk\ggmh$ is an isomorphism. 
\end{theorem}

\begin{proof}
According to the main theorem of Mac Lane--Vaqui\'e \cite[Theorem 4.3]{MLV}, we may distinguish two cases:

\begin{enumerate}
\item [(a)] The valuation $\mu$ can be obtained as an ordinary augmentation $\mu=[\eta;\phi,\infty]$ of some valuation-transcendental $\eta$. 
\item [(b)] The valuation $\mu$ can be obtained as a limit augmentation $\mu=[\cc;\phi,\infty]$ of some continuous family $\cc$.
\end{enumerate}

In case (a), as in every ordinary augmentation, we have
\[
\op{Ker}\left(\gg_\eta\to\ggm\right)=(\inm\phi)\gg_\eta,\qquad 
\im\left(\gg_\eta\to\ggm\right)=\ggm^0. 
\]  	

By Lemma \ref{Ordinary}, we have
 $\muh=[\eta^h;Q,\ga']$, for $Q=Q_{\eta,\phi}$. We deduce a commutative diagram of exact sequences of  graded algebras:
$$
\ars{1.2}
\begin{array}{ccccccc}
0\ \lra&\left(\inn_{\eta^h}Q\right)\gg_{\eta^h}&\lra&\gg_{\eta^h}&\lra&\ggmh^0=\ggmh&\lra\ 0\\
&\uparrow&&\uparrow&&\uparrow&\\
0\ \lra&\left(\inn_{\eta}\phi\right)\gg_{\eta}&\lra&\gg_{\eta}&\lra&\ggm^0=\ggm&\lra\ 0\\
\end{array}
$$ 
By Theorem \ref{theoremabouthensel}, the two first vertical arrows are isomorphisms. Hence, the third vertical arrow is an isomorphism too. 

In case (b), \cite[Corollary 5.6]{caiospiv} shows that 
\[
\ggm=\ggm^0=\lim_{i\in A} \gg^0_{\rho_i},\qquad \ggmh=\ggmh^0=\lim_{i\in A} \gg^0_{\rho_i^h}.
\]
Since for every $i\in A$ the canonical embedding $\mathcal G_{\rho_i}\hk \mathcal G_{\rho_i^h}$ is an isomorphism and these isomorphims commute with the homomorphisms $\gg_{\ri}\to\gg_{\rj}$ and $\gg_{\ri^h}\to\gg_{\rj^h}$, we deduce that $\ggm\hk\ggmh$ is an isomorphism. 
\end{proof}

\subsection{Rigidity of minimal pairs}
Although we do not need it for the defect formula, we include another interesting application of Theorem \ref{theoremabouthensel}. Fix an extension $\overline v$ of $v$ to $\overline K$. For $(a,\delta)\in\kb\times \La$, we denote by $\omega_{a,\delta} $ the valuation on $\overline K[x]$ defined by
\[
\omega_{a,\delta}\left(\sum_{n\geq 0}a_n(x-a)^n\right)=\min_{n\geq 0}\{\overline v(a_n)+n\delta\}.
\]
This valuation is called a \textbf{depth-zero} extension of $v$ and sometimes denoted by $[v;x-a,\delta]$.
The pair $(a,\delta)$ is a $(K,v)$-\textbf{minimal pair} if $a$ has the smallest degree over $K$ among all $a'\in\kb$ for which there exists $\delta'\in\La$ such that
\[
\left(\om_{a,\dta}\right)_{\mid\kx}=\left(\om_{a',\dta'}\right)_{\mid\kx}.
\]

Let us recall \cite[Propositions 2.8, 3.3]{Rig}:

\begin{proposition}\label{UpDown}\mbox{\null}
	
	\begin{enumerate}
		\item If $(a,\dta)\in \kb\times \La$ is a $(K,v)$-minimal pair, then the minimal polynomial of $a$ over $K$ is a key polynomial of minimal degree of $\left(\om_{a,\dta}\right)_{\mid\kx}$. 
		\item Let $\mu$ be a valuation-transcendental valuation. Take $\phi\in\kpm$ of minimal degree and denote 
\[
\dta=\epm(\phi)=\max_{b\in Z(\phi)}\{\overline \mu(x-b)\}
\]
where $\overline\mu$ is an arbitrary extension of $\mu$ to $\overline K[x]$.
 Then, for all $a\in Z(Q_{\mu,\phi})$, the pair $(a,\dta)$ is a $(K,v)$-minimal pair and $\mu=\left(\om_{a,\dta}\right)_{\mid\kx}$. 
	\end{enumerate} 
\end{proposition}

The first statement is classical \cite{PP} and can be found in \cite{N2019} as well.

\begin{theorem}\label{minimalitythem}
	If $(a,\dta)\in \kb\times \La$ is a $(K,v)$-minimal pair, then it is a $(\kh,\vh)$-minimal pair.
\end{theorem}

\begin{proof}
	Let $(a,\dta)$ be a $(K,v)$-minimal pair and denote
	$$\om=\om_{a,\dta},\qquad \mu=\om_{\mid \kx},\qquad \muh=\om_{\mid \khx},\qquad Q=Q_{\mu,\phi}.
	$$
	By item 1 in Proposition \ref{UpDown}, the minimal polynomial $\phi\in\kx$ of $a$ over $K$ is a key polynomial of minimal degree for $\mu$. By \cite[Proposition 3.3]{Rig}, we have
\[
\left(\om_{a,\dta}\right)_{\mid \kx}=\mu \ \sii\ a\in Z(Q). 
\]	
Hence, $Q$ is the minimal polynomial of $a$ over $\kh$.  By Proposition \ref{P3.3Isom},  $Q$ is a key polynomial for $\muh$ of minimal degree.
 Hence, the second item in  Proposition \ref{UpDown} shows that  $(a,\dta)$ is a $(\kh,\vh)$-minimal pair
\end{proof}

\begin{corollary}\label{corollaryandrei}
Let $\mu$ be valuation-transcendental and take $\phi\in {\rm KP}(\mu)$ of minimal degree. Then $\phi$ is irreducible in $K^h[x]$ if and only if $\phi\in {\rm KP}(\mu^h)$ if and only if $\deg(\mu)=\deg(\mu^h)$.
\end{corollary}
\begin{proof}
Let $a$ be a root of $Q:=Q_{\mu,\phi}$. Proposition \ref{UpDown} (2) implies that $(a,\delta)$ is a $(K,v)$-minimal pair. Hence, by Theorem \ref{minimalitythem}, it is also a $(K^h,v^h)$-minimal pair. Hence the minimal polynomial $Q$ of $a$ over $K^h$ is a key polynomial of minimal degree for $\mu^h$. Hence, all the statements of the corollary are equivalent to $Q=\phi$.
\end{proof}

\section{The defect formula}\label{Defectform}
\subsection{The defect of an augmentation}\label{defectofaaug}

The following result was proved by Herrera-Olalla-Spivakovsky in the rank-one case, and by Vaqui\'e \cite{Vaq0} in the general case.\e

\begin{lemma}\label{Lemma1} Let $\mu\to\nu$ be an augmentation. Let $\ty=\ty(\mu,\nu)$ be the corresponding tangent direction, equipped with the pre-ordering determined by the action of $\nu$. For $Q\in \ty$ set $\rho_Q=[\mu;Q,\nu(Q)]$. Let $\phi\in\kx$ be either a key polynomial of minimal degree of $\nu$, or $\supp(\nu)=\phi\kx$. 
Then, the positive integer
$$
d=\min\left\{\deg_{\rho_Q}(\phi)\mid Q\in\ty\right\}
$$
is independent of the choice of $\phi$. Moreover, the set of all $Q\in\ty$ such that $\deg_{\rho_Q}(\phi)=d$ is cofinal in $\ty$.
\end{lemma}

\begin{definition}
This stable value $d$ is called the \textbf{defect} of the augmentation.  We denote it by  $d=d(\mu\to\nu)$.
\end{definition}

\begin{lemma}\label{d=1} 
If $\mu\to\nu$ is an ordinary augmentation, then	$d(\mu\to\nu)=1$.
\end{lemma}

\begin{proof}
If $\nu=[\mu;\phi,\ga]$, then $\ty(\mu,\nu)=[\phi]_\mu$. Since $\phi$ is a key polynomial of minimal degree for $\nu=\rho_\phi$, we have $\deg_{\rho_\phi}(\phi)=1$. Hence, $d(\mu\to\nu)=1$.
\end{proof}

\begin{lemma}\label{Lemma3} 
Suppose $(K,v)$ henselian. For all augmentations $\mu\to\nu$ we have 
$$d(\mu\to\nu)=\deg(\nu)/\deg\ty.$$
In particular, if $\mu\to\nu$ is a limit augmentation, then 
$d(\mu\to\nu)=\deg(\nu)/\deg(\mu)$.
\end{lemma}

\begin{proof}
	If  $\supp(\nu)=\phi\kx$, then we have directly $\nu=v_\phi$, by Proposition \ref{CorEndler2}.
	
If $\phi$ is a key polynomial of minimal degree for $\nu$, then $\nu<[\nu;\phi,\infty]=v_\phi$. 

In both cases, $\mu<v_\phi$ and $\ty(\mu,v_\phi)=\ty(\mu,\nu)=[Q]_\mu$, for some $Q\in\kpm$.
 Since $\phi$ is irreducible in $\kx$, Theorem \ref{fundamental} shows that $\deg_{\rho_Q}(\phi)=\deg(\phi)/\deg(Q)$. By definition, $\deg(\nu)=\deg(\phi)$ and $\deg\ty=\deg(Q)$. 

If  $\mu\to\nu$ is a limit augmentation, then $\deg\ty=\deg(\mu)$ by our convention on limit augmentations (cf. Section \ref{subsecLimAug}).
\end{proof}

\subsection{Rigidity of defect}\label{rigofdefect}

We proceed now with the proof of Theorem \ref{defectcte}.\e

%\begin{proof}

\nn{\bf Proof of Theorem \ref{defectcte}.}
If $\mu\to\nu$ is an ordinary augmentation, then $\muh\to\nu^h$ is an ordinary augmentation too, by Lemma \ref{Ordinary}. By Lemma \ref{d=1}, $d(\mu,\nu)=1=d(\muh,\nu^h)$.
	
	Suppose now that $\mu\to\nu$ is a limit augmentation; that is, $\nu=[\cc;\phi,\ga]$ for some continuous family $\cc$ based on $\mu$. 
	
	Let $Q=Q_{\cc,\phi}$.
	 By Lemma \ref{unstabilitythm}, $\nu^h=[\cc^h;Q,\ga']$ and $\phi/Q$ is $\cc^h$-stable. Consider an index $i_0\in A$ such that $\inn_{\ri^h}(\phi/Q)$ is a unit in $\gg_{\ri^h}$ for all $i\ge i_0$.

	 Since the canonical embedding $\gg_{\ri}\hk\gg_{\ri^h}$ is an isomorphism for all $i\in A$, it preserves degrees of homogeneous elements. Hence,
\[
\deg_{\ri}(\phi)=\deg_{\ri^h}(\phi)=\deg_{\ri^h}(Q)\quad\mbox{ for all }i\ge i_0.
\]
By the definition of the defect of an augmentation,
\[
\quad\qquad\qquad d(\mu\to\nu)=\min_{i\in A}\deg_{\ri}(\phi)=\min_{i\in A}\deg_{\ri^h}(Q)=d(\muh\to\nuh). \qquad\qquad\quad\Box
\]
%\end{proof}

%\begin{corollary}\label{sameesfs}In the above situation, $\mu$, $\muh$ have the same numerical invariants $e_0,\dots, e_r$ and $f_0,\dots, f_r$  defined in (\ref{esfs}).\end{corollary}

\subsection{Rigidity of inertia degree}\label{subsecInertia}

We define the \textbf{inertia degree} of an augmentation $\mu\to\nu$ as
\[
f(\mu\to\nu)=[\ka_\nu\colon \ka_\mu],
\]
where $\ka_\mu$ is the relative algebraic closure of $k$ inside the grade-zero component $\d_\mu$ of the graded algebra $\ggm$. 

As an immediate consequence of Theorem \ref{theoremabouthensel}, we get
\[
f(\mu\to\nu)=f(\muh\to\nu^h).
\]
Indeed, we have a commutative diagram of field extensions
\[
\ars{1.2}
\begin{array}{ccc}
\ka_{\muh}&\lra&\ka_{\nu^h}\\
\uparrow&&\uparrow\\
\ka_{\mu}&\lra&\ka_{\nu}
\end{array}
\]
and the vertical arrows are isomorphisms.

Let us quote a direct computation of the inertia degree of an augmentation, \cite[Lemma 5.2]{MLV}:
\begin{equation}\label{Compf}
f(\mu\to\nu)=
\dfrac{\deg\ty(\mu,\nu)}{e(\mu)\deg(\mu)}.
\end{equation}

Note that  $f(\mu\to\nu)=1$, if $\mu\to\nu$ is a limit augmentation. Indeed, by our convention on limit augmentations, we have $e(\mu)=1$ and $\deg\ty(\mu,\nu)=\deg(\mu)$. 

From the rigidity of defect and inertia degree we derive a relevant consequence.

\begin{proposition}\label{degG}
For every augmentation $\mu\to\nu$, we have
\[
e(\mu)f(\mu\to\nu)d(\mu\to\nu)=\deg(\nu^h)/\deg(\muh).
\]
\end{proposition}

\begin{proof}
If  $\mu\to\nu$ is ordinary, then $d(\mu\to\nu)=1$ by Lemma \ref{d=1}. By Lemma \ref{Ordinary}, the augmentation $\muh\to\nu^h$ is ordinary too. If we apply (\ref{Compf}) to this augmentation, then we get
\[
f(\mu\to\nu)=f(\muh\to\nu^h)=\dfrac{\deg\ty(\muh,\nu^h)}{e(\muh)\deg(\muh)}=\dfrac{\deg(\nu^h)}{e(\mu)\deg(\muh)},
\]
where the equality $e(\muh)=e(\mu)$ follows from Theorem \ref{theoremabouthensel}.

If  $\mu\to\nu$ is limit, then $f(\mu\to\nu)=1$ and $e(\mu)=1$. The latter equality, by our general convention on limit augmentations (see Section \ref{subsecLimAug}).

In this case, the statement follows from Theorem \ref{defectcte} and Lemma \ref{Lemma3}.
\end{proof}

\subsection{Proper augmentations}\label{subsecProper}
A key polynomial $\phi\in\kpm$ is \textbf{proper} if there exists some key polynomial $\varphi\in\kpm$ of minimal degree  
such that $\phi\not\smu\varphi$.

Clearly, this condition extends to $\mu$-equivalence classes
of key polynomials, so that we may speak of proper and improper classes.

Let $e:=e(\mu)$ be the relative ramification index of $\mu$. By \cite[Proposition 6.3]{KP}, 
\begin{itemize}
	\item If $e=1$, then all key polynomials for $\mu$ are proper. 
	\item If $e>1$, then there is a unique improper class, formed by all key polynomials for $\mu$ of minimal degree. 
\end{itemize}  

If $\mu$ is value-transcedental, then all key polynomials for $\mu$ are in the same $\mu$-equivalence class: $\kpm=[\phi]_\mu$ \cite[Theorem 4.2]{KP}. In other words:

\begin{lemma}\label{AllImproper}
If $\mu$ is value-transcendental, then all its key polynomials are improper.
\end{lemma}

\begin{lemma}\label{wholeGroup}
	If $\phi$ is a proper key polynomial for $\mu$, then
	\[
	\g_{\mu,\phi}:=\left\{\mu(a)\mid a\in\kx,\ 0\le \deg(a)<\deg(\phi)\right\}=\gm.
	\]
\end{lemma}

\begin{proof}
	Clearly, $\gm^0\subset\g_{\mu,\phi}\subset\gm$. If $e:=e(\mu)=1$, then the three groups in this chain coincide.
	Suppose $e>1$, and let $\varphi$ be a key polynomial for $\mu$ of minimal degree. As mentioned in Section \ref{subsecDegree}, $\gm=\gen{\gm^0,\mu(\varphi)}$. Since $\phi$ is proper, necessarily $\deg(\phi)>\deg(\varphi)$. Hence, $\gm=\gen{\gm^0,\mu(\varphi)}\subset \g_{\mu,\phi}$.
\end{proof}

\begin{definition}
	An augmentation $\mu\to\nu$ is \textbf{proper} if $\ty(\mu,\nu)$ is a proper class. 
\end{definition}

\begin{proposition}\label{properAugm}
	If $\mu\to\nu$ is a proper augmentation, then $\gm=\gn^0$.
\end{proposition}

\begin{proof}
	Let $e=e(\mu)$. If $e=1$, then we have $\gm=\gm^0\subset\gn^0$ by Lemma \ref{00}.

	Suppose that $e>1$. Take any $\phi\in\ty(\mu,\nu)$ and consider the ordinary augmentation $\eta=[\mu;\phi,\nu(\phi)]$. Clearly, $\mu<\eta\le \nu$.
	
	Since $\phi$ is proper, Lemma \ref{wholeGroup} shows that $\gm=\g_{\mu,\phi}$. Hence, any $\al\in\gm$ can be written as $\al=\mu(a)$ for some $a\in\kx$ such that $\deg(a)<\deg(\phi)$.
	
	Then, $\inm \phi\nmid \inm a$, and Lemma \ref{tdef} shows that $\mu(a)=\nu(a)=\eta(a)$. Since $\phi$ is a key polynomial for $\eta$ of minimal degree, we have $\al=\eta(a)\in\geta^0\subset \gn^0$, the last inclusion by Lemma \ref{00}.
	This ends the proof of the inclusion $\gm\subset\gn^0$.
	
	Conversely, take $\al\in\gn^0$. If the augmentation $\mu\to\nu$ is ordinary, then we can write $\al=\nu(a)$ for some $a\in\kx$ with $\deg(a)<\deg(\phi)=\deg(\nu)$. Then, $\inm \phi\nmid \inm a$, and Lemma \ref{tdef} shows that $\al=\nu(a)=\mu(a)\in\gm$. Thus, $\gn^0\subset\gm$.
	
	If the augmentation $\mu\to\nu$ is limit, then $\nu=[\cc;\phi,\ga]$  for some continuous family $\cc$ based on $\mu$. Then,
	$$
	\gn^0=\{\rho_\cc(a)\mid a\in\kx,\ 0\le\deg(a)<\deg(\phi)\}\subset\g_\cc.
	$$
	By our convention about limit augmentations we have $\g_\cc=\gm$. 
\end{proof}

\begin{lemma}\label{RigProper}
 An augmentation $\mu\to\nu$ is proper if and only if the augmentation $\muh\to\nuh$ is proper.  
\end{lemma}

\begin{proof}
Since $\mu<\nu$, we have $\kpm\ne\emptyset$, so that $\mu$ is necessarily valuation-transcendental. By Lemma \ref{AllImproper}, if $\mu$ is value-transcendental it admits no proper augmentations. 
Since the same argument applies to $\muh$, we can assume that $\mu$ (hence $\muh$) is residue-transcendental. 

Let $e=e(\mu)=e(\muh)$. If $e=1$, then all key polynomials are proper and the statement of the lemma is obvious. 

Suppose that $e>1$. In this case, all key polynomials of minimal degree are $\mu$-equivalent and form the unique improper class of $\mu$. Thus, the prime element  $\pi=\inm\varphi$  does not depend on the choice of $\varphi$ among the key polynomial for $\mu$ of minimal degree. Consider the improper homogeneous prime ideal $\pp=\pi\ggm$.

The subring $\d=\dm\subset\ggm$ formed by all homogeneous elements of grade zero is a polynomial ring with coefficients in the field $\ka=\ka_\mu$:
$$
\d=\ka[\xi],\qquad \xi=\pi^e/u,
$$
for any  homogeneous unit $u$ of grade $-e\mu(\phi)$. The element $\xi$ is transcendental over $\ka_\mu$ \cite[Theorem 4.6]{KP}.

The canonical prime ideal $\p_0=\xi\d$ ramifies in $\ggm$; more precisely, $\p_0\ggm=\pp^e$. However, all other prime ideals of $\d$ remain prime in $\ggm$ \cite[Lemma 6.1]{KP}.

Hence, the improper ideal $\pp$ is characterized by an algebraic property: any generator of $\pp$ raised to the $e$-th power generates a prime ideal in $\d$. 

As a consequence, the isomorphism $\ggm\hk\ggmh$ must send the improper ideal of $\ggm$ to the improper ideal of $\ggmh$.  

Finally, let $\ty(\mu,\nu)=[\phi]_\mu$, $\ty(\muh,\nuh)=[Q]_{\muh}$. By Lemmas \ref{Ordinary} and \ref{unstabilitythm}, the isomorphism $\ggm\hk\ggmh$ maps the prime ideal $\left(\inm\phi\right)\ggm$ to the prime ideal $\left(\inmh Q\right)\ggmh$. This ends the proof.  
\end{proof}

\subsection{Proper chains of augmentations}\label{subsecProperChains}

Suppose that  $\mu$ is either valuation-trans\-cen\-den\-tal or has nontrivial support.

By the celebrated result of Mac Lane--Vaqui\'e, $\mu$ can be obtained from $v$ after a finite number of augmentations: 
\begin{equation}\label{MLV}
v\ \to\  \mu_0\ \to\  \mu_1\ \to\ \cdots
\ \to\ \mu_{r} 
\ \to \ \mu_{r+1}=\mu.
\end{equation}

The initial step $v\to\mu_0$ has only a formal purpose, because the valuation $v$ is not a valuation on $\kx$. This step only indicates that $\mu_0$ is a depth-zero valuation $\mu_0=\om_{a,\dta}$, for some $a\in K$, $\ga\in\La\infty$ (see Section 6.3).\e

As shown in \cite[Theorem 4.3]{MLV}, we can impose some strong conditions on the chain, ensuring a certain kind of uniqueness. The chains satisfying these conditions are called  \emph{MLV-chains}. However, these chains are not useful to our purposes, because the MLV-condition does not (necessarily) hold any more when we  extend the chain to the henselization (see Examples \ref{exa1} and \ref{exampletwoaug}).

\emph{Proper} chains of augmentations are best suited.

\begin{definition}
A chain of augmentations as in (\ref{MLV}) is said to be \textbf{proper} if all augmentations $\mu_n\to\mu_{n+1}$ are proper.  
\end{definition}

All MLV-chains are proper \cite[Section 4.1]{MLV}. Thus, our valuations $ \mu$ can be obtained as the final node of a proper chain. Also, proper chains remain proper under henselization, by Lemma \ref{RigProper}.   

Obviously, by considering the canonical homomorphisms $\gg_{\mu_n}\to\gg_{\mu_{n+1}}$, any chain induces a tower of fields
\[
k=\ka_{\mu_0}\,\to\,\cdots\,\to\,\ka_{\mu_{r+1}}=\ka_\mu.
\]
Hence, we have
\[
[\ka_\mu\colon k]=f(\mu_0\to\mu_1)\cdots f(\mu_r\to\mu).  
\]

Also, if the chain is proper, then Proposition \ref{properAugm} shows that the value groups of the nodes of the chain form a chain 
\[
\g_{\mu_{-1}}:=\g\subset\g_{\mu_0}\subset \cdots \subset \g_{\mu_{r}}\subset \g_{\mu_{r+1}}= \gm,
\]
such that $\g_{\mu_{n-1}}=\g_{\mu_n}^0$ for all $n$, $0\le n\le r+1$.
Hence, we have
\[
\left(\gm\colon\g\right)=e(\mu_0)\cdots e(\mu_{r+1}). 
\]

Finally, let us denote
\[
d(\mu/v):=d(\mu_0\to\mu_1)\cdots d(\mu_r\to\mu). 
\]

Since $\deg(\mu_0)=1$, Proposition \ref{degG} shows  that 
\[
\left(\gm\colon\g\right)[\ka_\mu\colon k]d(\mu/v)=e(\mu)\deg(\muh).
\]
Hence the following definition is correct, and Theorem \ref{degGthm} follows  immediately.

\begin{definition}%\label{dintrinsic}
The integer $d(\mu/v)$ does not depend on the choice of the proper chain ending in $\mu$.  We call it the defect of $\mu/v$. 
\end{definition}

\begin{theorem}\label{degGthm}
Suppose that $\mu$ is valuation-transcendental or has nontrivial support. Then,
 \[
  \left(\gm\colon\g\right)[\ka_\mu\colon k]d(\mu/v)=e(\mu)\deg(\muh).
 \]
\end{theorem}

\subsection{The defect formula}

Let $g\in\irr(K)$ and let $w$ be an extension of $v$ to the field $L=\kx/(g)$.

This valuation $w$ on $L$ determines a valuation on $\kx$ in an obvious way:
\[
\mu\colon \kx\longtwoheadrightarrow L\stackrel{w}\lra \gq\infty.
\]
This valuation has support $g\kx$. As shown in Section \ref{secFinLeaves}, we have $\mu=w_G$ for some irreducible factor $G\in \irr(\kh)$ of $g$ in $\khx$. According to the notation in that section, we have $\muh=v_G$.

By Theorem \ref{degGthm}, we deduce that
\[
 e(w/v)f(w/v)d(\mu/v)=\deg(\muh)=\deg(G).
\]
Indeed, we have obviously $\gm=\g_w$. Also, $e(\mu)=1$ because the algebra $\ggm$ is simple: $\ggm=\ggm^0$. On the other hand, $k_w=k_\mu$ by definition, and it is easy to check that $\ka_\mu=k_\mu$.  
We deduce our main result.

\begin{theorem}\label{DF}
 Let $w$ be an extension of $v$ to a finite simple extension $L/K$. Let $\mu$ be the valuation on $\kx$ induced by $w$.    
For any proper chain of augmentations ending in $\mu$ as in (\ref{MLV}), we have  
\[
d(w/v)=d(\mu_0\to\mu_1)\cdots d(\mu_r\to\mu).
\]
\end{theorem}

In particular, $d(w/v)$ is a positive integer. The next result follows from Theorems  \ref{defectcte}, \ref{theoremabouthensel} and 
\ref{DF}.

\begin{corollary}\label{efdHensel}
 Let $w$ be an extension of $v$ to a finite simple extension $L/K$. Let $w^h$ be the natural extension of $\vh$ to $L^h=L \kh$ determined by $w$. Then,
 \[
 e(w/v)=e(w^h/v^h),\quad 
 f(w/v)=f(w^h/v^h),\quad 
 d(w/v)=d(w^h/v^h). 
 \]
\end{corollary}
In particular, Ostrowski's lemma shows that $d(w/v)$ can be expressed as a power of the residue characteristic exponent of $v$.

\begin{corollary}\label{efdHensel2}
Suppose that $(K,v)$ is henselian. Let $w$ be an extension of $v$ to a finite simple extension $L/K$. Let $\mu$ be the valuation on $\kx$ determined by $w$.  For any proper chain of augmentations ending in $\mu$ as in (\ref{MLV}), we have
\[
d(w/v)=\prod_{n\in I_{\lim}}\dfrac{\deg(\mu_{n+1})}{\deg(\mu_n)},
\]
where the set $I_{\lim}$ contains all indices $n$ such that $\mu_n\to\mu_{n+1}$ is a limit augmentation.
\end{corollary}

This formula for the defect in the henselian case was obtained by Vaqui\'e \cite{Vaq3}.

\section{The rank one case}\label{rankone}

All extensions $\mu$ of $v$ to $\kx$ considered in this section are assumed to satisfy ${\rm rk}(v)={\rm rk}(\mu)=1$.
By Abhyankar's inequality, the corresponding extensions $\muh$ to $\khx$ have rank one too. 

 Under these conditions, it is much easier to prove that the map $\mathcal G_\mu\hk \mathcal G_{\mu^h}$ is an isomorphism.

\begin{lemma}\label{lemmarankone}
For every $F\in K^h[x]\setminus\supp(\muh)$ there exists $f\in K[x]\setminus\supp(\mu)$, with same degree as $F$, such that $F\sim_{\muh}f$. Moreover, if $F\in\kp(\muh)$, then $f\in\kp(\muh)$.
\end{lemma}

\begin{proof}
Write $F=F_0+\ldots+F_rx^r$, with $F_i\in\kh$. Since ${\rm rk}(v)={\rm rk}(\muh)=1$, for each $i$, $0\leq i\leq r$, there exists $f_i\in K$ such that $v^h(F_i-f_i)>\muh(F)-i\nu(x)$. Set $f=f_0+f_1x+\ldots+f_rx^r$. Then
\[
\muh(F-f)\geq \min_{0\leq i\leq r}\{\muh\left((F_i-f_i)x^i\right)\}>\muh(F).
\] 
The last statement follows from \cite[Lemma 2.5]{KP}.
\end{proof}

\begin{corollary}\label{Thm1.3}
The map $\mathcal G_\mu\hk \mathcal G_{\mu^h}$ is an isomorphism.
\end{corollary}

\begin{proof}
Since $\mathcal G_\mu\hk \mathcal G_{\mu^h}$ is injective, it is enough to show that it is surjective. This follows immediately from Lemma \ref{lemmarankone}.
\end{proof}\e

%In order to proof Theorem \ref{maintehmrank1}, we need some auxiliary results.

Now, let us show that  the rigidity statements for augmentations $\mu\to\nu$ take a very special form, with one exception: the limit augmentations in which $\supp(\nu)\ne0$. 

\begin{lemma}\label{kp=kp}
If $\kpm\ne\emptyset$, then $\kp(\muh)\cap \kx=\kpm$. In particular, all key polynomials for $\mu$ are irreducible in $\khx$.
\end{lemma}

\begin{proof}
By Corollary \ref{Thm1.3}, for all $f\in\kx\setminus\supp(\mu)$, the homogeneous element $\inv_\mu f$ is prime in $\ggm$ if and only if $\inn_{\muh}f$ is prime in $\ggmh$.	
	
Take $Q\in \kp(\muh)\cap \kx$. By Corollary \ref{Thm1.3}, $\inv_\mu Q$ is a prime element. Also,  the $\mu$-minimality of $Q$ follows trivially from  its $\muh$-minimality. Hence, $Q\in\kpm$. %This proves $\kp(\nuh)\cap \kx\subset\kpn$.

Conversely, take $\phi\in\kpm$. By Corollary \ref{Thm1.3}, $\inn_{\muh}\phi$ is a prime element. Let us show that $\phi$ is $\muh$-minimal. For $F\in\khx\setminus\supp(\muh)$, suppose that $\inn_{\muh}\phi\mid \inn_{\muh}F$. By Lemma \ref{lemmarankone}, $\inn_{\muh}\phi\mid \inn_{\muh}f$, for some $f\in \kx\setminus \supp(\mu)$ such that $\deg(f)=\deg(F)$. By Corollary \ref{Thm1.3}, $\inn_\mu\phi\mid \inv_\mu f$, so that $\deg(\phi)\le\deg(f)=\deg(F)$, by the $\mu$-minimality of $\phi$. Hence, $\phi$ is $\muh$-minimal and $\phi\in\kp(\muh)$.
\end{proof}

\begin{corollary}\label{deg=deg}
If $\kpm\ne\emptyset$, then $\deg(\mu)=\deg(\muh)$.
\end{corollary}

\begin{proof}
Let $\phi\in\kx$ be a key polynomial for $\mu$ of minimal degree. By Lemma \ref{kp=kp}, $\phi$ is  a key polynomial for $\muh$. By Lemma \ref{lemmarankone}, there cannot exist $F\in\kp(\muh)$ with $\deg(F)<\deg(\phi)$, because then it would exist $f\in\kpm$ with $\deg(f)=\deg(F)$, contradicting the minimality of $\deg(\phi)$ in $\kpm$.	
\end{proof}

\begin{lemma}\label{OrdAug}
	Let $\nu=[\mu;\phi,\ga]$ be an ordinary augmentation. 	Then, $\phi\in\kp(\muh)$ and $\nuh=[\muh;\phi,\ga]$.

\end{lemma}

\begin{proof}
By Lemma \ref{kp=kp}, $\phi\in\kp(\muh)$. Hence, its irreducible factor $Q_{\mu,\phi}$ is equal to $\phi$ itself. The result follows from Lemma \ref{Ordinary}.
\end{proof}

\begin{lemma}\label{kpi=kpi}
Let $\nu=[\cc;\phi,\ga]$ be a limit augmentation such that $\ga<\infty$. Then, $\phi\in\kpi(\cc^h)$ and $\nuh=[\cc^h;\phi,\ga]$.
\end{lemma}

\begin{proof}
Since $\ga<\infty$, we have $\phi\in\kpn$. By Lemma \ref{kp=kp}, $\phi\in\kp(\nuh)$.  Hence, its irreducible factor $Q_{\cc,\phi}$ is equal to $\phi$ itself. The result follows from Lemma \ref{unstabilitythm}.
\end{proof}

\begin{corollary}\label{def=def}
	Let $\mu\to\nu$ be a limit augmentation such that $\supp(\nu)=0$. Then,
	\[
	d(\mu\to\nu)=\deg(\nu)/\deg(\mu).
	\]	
\end{corollary}

\begin{proof}
	By Theorem 1.2 and Proposition 7.4,
	\[
d(\mu\to\nu)=d(\muh\to\nuh)=\deg(\nuh)/\deg(\muh).	
	\]
The result follows from Corollary 8.4.
\end{proof}\e

Nevertheless, for a limit augmentation $\mu\to\nu$ where $\supp(\nu)\ne0$, we have
\[
d(\mu\to\nu)=\deg(\nuh)/\deg(\muh)=\deg(\nuh)/\deg(\mu).
\] \bs

\section{Examples}
We start this section by presenting examples of key polynomials for valuations on $K[x]$ that are not irreducible in $K^h[x]$.  These are counterexamples for the final conjecture in \cite{Andrei}.

\begin{lemma}\label{lemmaeveryirrediskey}
Let $v$ be a valuation on $K$ and take $g\in \irr(K)$. Then there exists a valuation on $K[x]$, extending $v$, for which $g$ is a key polynomial.
\end{lemma}
\begin{proof}
Let $L=K[x]/(g)=K(\theta)$ be the simple extension defined by $g$. Let $w$ be any valuation on $L$ extending $v$. Then, by Theorem 1.1 of \cite{N2021}, the map
\[
\mu(f_0+f_1g+\ldots+f_rg^r):=\min_{0\leq i\leq r}\{(i,w(f_i(\theta))\}\in(\Z\times_{\rm lex}\gq)\infty
\]
is a valuation on $K[x]$ (extending $v$). It is not difficult to show that $g$ is a key polynomial for $\mu$ of minimal degree.
\end{proof}
\begin{remark}
The valuation above can seen as an augmentation of $w$ defined by $\mu(g)=(1,0)$.
\end{remark}

 The next example shows that key polynomials for a valuation on $K[x]$ do not need to be irreducible in $K^h[x]$.
\begin{example}\label{exampnaphel}
Let $(K,v)$ be a non-henselian field. Then, there exists $a\in K^h\setminus K$. Let $g$ be the minimal polynomial of $a$ over $K$. By Lemma \ref{lemmaeveryirrediskey}, there exists a valuation $\mu$ on $K[x]$ such that $g$ is a key polynomial. However, since $a\in K^h$ is a root of $g$, we deduce that $g$ is not irreducible in $K^h[x]$. 
\end{example}
\begin{example}\label{exa1}%\label{chiannotmaclane}

Let $(K,v)$ be a valued field such that $p:={\rm char}(K)>0$. Assume that $a\in K$ is such that $g=x^p-x-a\in K[x]$ is irreducible and that $a^{\frac{1}{p^i}}\in K$ for every $i> 0$. Set
\[
S=\left\{\frac{v(g(b))}{p}\mid b\in K\right\}\subseteq \frac 1p vK.
\]
We will discuss two cases.\\
\textbf{Case 1.} Assume that $S<0$, $S\subseteq vK$ and $S$ does not have a last element (this can be realized taking $K=\F_p(t)^{\frac{1}{p^\infty}}$ with the $t$-adic valuation and $a=t^{-1}$). In this case (see \cite[Theorem 1.6]{BN}) there is a unique extension $w$ of $v$ to $K[x]/(g)$ and $d(w/v)=p$. Let $\mu$ be the corresponding valuation on $K[x]$ with support $gK[x]$. Set
\[
a_n=\sum_{i=1}^n a^{\frac{1}{p^i}}\in K.
\]
Consider the depth-zero valuation $\mu_0=\left[v; x,\frac{v(a)}{p}\right]$ and the family of valuations $\mathcal C=(\rho_n)_{n\in \N}$ defined by $\rho_n=\left[\mu_0;x-a_n, \frac{v(a)}{p^{n+1}}\right]$. One can show that this family is a continuous family of valuations (supported on $\N$), $g$ is a limit key polynomial for $\mathcal C$ and that $\mu=[\mathcal C;g,\infty]$. Also, for each $n\in\N$ we have
\[
g=(x-a_n)^p-(x-a_n)-g(a_n).
\]
Since $\rho_n(x-a_n)<0$, this implies that $v(g(a_n))=\rho_n\left((x-a_n)^p\right)<\rho_n(x-a_n)$. Thus, $S_{\rho_n,x-a_n}(g)=\{0,p\}$ and $\deg_{\rho_n}(g)=p$. In particular,
\[
v\ \stackrel{x,\frac{v(a)}{p}}\lra\  \mu_0\ \stackrel{g,\infty}\lra\mu.
\]
is a MLV chain for $\mu$. Fix a henselization $K^h$ of $K$ and an extension $v^h$ of $v$ to $K^h$. One can show that
\[
\mu^h_0:=\left[v^h; x,\frac{v(a)}{p}\right]\mbox{ and }\rho^h_n=\left[\mu^h_0;x-a_n, \frac{v(a)}{p^{n+1}}\right]\mbox{ for }n\in\N,
\]
defines a continuous family $\mathcal C^h$ of valuations on $K^h[x]$. Also, this family is the image of $\mathcal C$ by the map $\mu\mapsto \mu^h$, $g$ is a limit key polynomial for $\mathcal C^h$ and $\mu^h=[\mathcal C^h;g,\infty]$. In particular,
\[
v\ \stackrel{x,\frac{v(a)}{p}}\lra\  \mu_0\ \stackrel{g,\infty}\lra\mu
\]
is a MLV chain for $\mu$ and
\[
p=d(w/v)=d(\mu_0\ra \mu)=\deg(\mu)/\deg(\mu_0)=\deg(\mu^h)/\deg(\mu_0^h).
\]\\
\textbf{Case 2.}  Assume now that $v(a)>0$. Then $v$ admits $p$ many distinct extensions to $K[x]/(g)$. They can be computed as follows. One can show that
\[
\theta_j:=j+\sum_{n=0}^\infty a^{p^n}, 0\leq j\leq p-1,
\]
are all the roots of $g$ (in the completion of $K$). Also, all these roots lie in the henselization of $K$ and hence $g$ is not irreducible in $K^h[x]$. The valuations on $K[x]/(g)$ extending $v$ are
\[
w_j(f+gK[x]):=\overline{v}(f(\theta_j))
\]
where $\overline v$ is a fixed extension of $v$ to an algebraic closure of $K$ (containing $\theta_0,\ldots,\theta_{p-1}$). Fix $w=w_0$, $\theta=\theta_0$ and $\mu$ the corresponding valuation on $K[x]$ with support $gK[x]$. In a similar way as in {\bf Case 1.} we can construct the valuations 
\[
\mu_0:=\left[v; x,v(a)\right]\mbox{ and }\rho_n=\left[\mu_0;x-a_n, p^{n+1}v(a)\right]\mbox{ for }n\in \N,
\]
where
\[
a_n=\sum_{i=0}^n a^{p^i}.
\]
As in {\bf Case 1.} the family $\mathcal C=(\rho_n)_{n\in \N}$ is a continuous family of valuations, $g$ is a limit key polynomial for $\mathcal C$ and $\mu=[\mathcal C;g,\infty]$. In particular,
\[
v\ \stackrel{x,v(a)}\lra\  \mu_0\ \stackrel{g,\infty}\lra\mu
\]
is a MLV chain for $\mu$. However, in this case $g$ is not a limit key polynomial for the corresponding family $\mathcal C^h$. It is easy to show that $x-\theta\in K^h[x]$ is a limit key polynomial for $\mathcal C^h$ and that $\mu^h=[\mathcal C^h;x-\theta,\infty]$. In particular,
\[
\deg(\mu)/\deg(\mu_0)=p>1=\deg(\mu^h)/\deg(\mu_0^h)=d(\mu_0^h\ra \mu^h).
\] 

Now, let us show directly (i.e., without using Theorem \ref{defectcte}) that $d(\mu_0\to\mu)=1$. For each $n\geq 0$ we have
\[
g=(x-a_n)^p-(x-a_n)-g(a_n).
\]
Since $\rho_n(x-a_n)>0$, this implies that $\rho_n\left((x-a_n)^p\right)>\rho_n(x-a_n)=v(g(a_n))$. Thus, $S_{\rho_n,x-a_n}(g)=\{0,1\}$ and $\deg_{\rho_n}(g)=1$. Consequently, $d(\mu_0\to\mu)=1$.

\end{example}
\begin{remark}
In \textbf{Case 2.} of the previous example, the chain
\[
v^h\ \stackrel{x-\theta,\infty}\lra\  \mu^h
\]
is a MLV chain for $\mu^h$. This also shows that $\mu\mapsto \mu^h$ does not preserve the ``depth" of valuations.
\end{remark}

The previous example presents a valuation $\mu$ on $K[x]$ whose MLV chain admits a limit augmentation but the corresponding valuation $\mu^h$ on $K^h[x]$ does not. The next example presents a valuation whose MLV chain needs two limit augmentations and the corresponding valuation on $K^h[x]$ needs none. Also, this example allows us to present a residue-transcendental valuation that admits a key polynomial that is not irreducible in $K^h[x]$. 
\begin{example}\label{exampletwoaug}
Consider the field $K=\Q(t)$ equipped with the $\ord_t$ valuation. Every $u\in K^*$ has an initial term 
$$
\inn(u)=\left(u\,t^{-\ord_t(u)}\right)(0)\in \Q^*.
$$

Take a prime number $p\equiv1\md4$ and let $\ord_p$ be the $p$-adic valuation.

Consider the following discrete rank-two valuation on $K$:
$$
v\colon K^*\lra \Z^2_{\lx},\qquad v(u)=\left(\ord_t(u),\ord_p(\inn(u))\right).
$$ 

%The residue field is $k=\F_p$.

For a fixed algebraic closure $\kb$, let $\vb$ be an extension of $v$ to $\kb$
and $(\kh,\vh)$ the henselization of $(K,v)$ determined by this choice. 

Let  $i\in\kb$ be a root  of the polynomial $x^2+1$. Consider its $p$-adic expansion
$$
i=i_0+i_1p^{\ell_1}+\cdots+ i_np^{\ell_n}+\cdots,
$$
with $0<i_n<p$ for all $n$. Denote the truncations of $i$ by
$$
a_n=i_0+i_1p^{\ell_1}+\cdots+ i_{n-1}p^{\ell_{n-1}}\in\Z.
$$

On the other hand, consider $\al\in\kb$ defined as
$$
\al=\sqrt{1+t}=1+(1/2)t+j_2t^2+\cdots +j_nt^n+\cdots,
$$
where $j_n\in\Q$ for all $n$. It is easy to check that $i$ and $\al$ belong to $\kh$.

The following observation is an easy exercise.
\begin{equation}\label{exercise}
\left\{\vb(i-a)\mid a\in K\right\}=\{0\}\times \Z.
\end{equation}

Consider the  algebraic element $\t=i+i\al\in\kh$. Clearly,
\begin{equation}\label{z}
\vb(\t-2i)=\vb(i(\al-1))=\vb(\al-1)=(1,0).
\end{equation}

Let $G=x-\t\in\khx$.
Our aim is to compute a MLV chain of the valuation 
$$
w_G=\left(v_G\right)_{\mid \kx}, \qquad v_G=[\vh;\,G,\infty].
$$
Although $v_G$ is a depth-zero valuation, its restriction $w_G$ has depth equal to two, because it will have a MLV chain of length two.

Note that $\supp(w_G)=g\kx$, where $g$ is the minimal polynomial of $\t$ over $K$: 
\[
g=x^4+(2t+4)x^2+t^2\in \kx. 
\]

For all $a\in K$, by (\ref{exercise}) and (\ref{z}) we deduce that
\begin{equation}\label{wa}
w_G(x-a)=\vb(\t-a)=\vb\left((\t-2i)+(2i-a)\right)=\vb(2i-a).
\end{equation}

Consider the continuous family of degree-one valuations:
$$
\cc=(\rho_n)_{n\ge0},\qquad \rho_n=[v; x-2a_n,(0,\ell_n)],
$$
based on the Gauss' valuation $\mu_0:=\rho_0=[v;x,(0,0)]$.

By (\ref{wa}), $w_G(x-2a_n)=\vb(2i-2a_n)=(0,\ell_n)$. Hence, $\rho_n<w_G$ for all $n$.

Also, it is easy to deduce from (\ref{wa}) that all polynomials of degree one are $\cc$-stable.
On the other hand, the polynomial $\phi=x^2+4$ is a limit key polynomial for $\cc$ because it is $\cc$-unstable. Indeed, since $v(a_n^2+1)=(0,\ell_n)$ and
\[
\phi=(x-2a_n)^2+4a_n(x-2a_n)+4+4a_n^2,
\]
we deduce that $\rho_n(\phi)=(0,\ell_n)$. Thus,  $\rho_n(\phi)$ grows with $n$.

Finally, $w_G(\phi)=\vb(\phi(\t))=(1,0)$, because
\[
\phi(\t)= \t^2+4=-2\left(t+j_2t^2+\cdots+j_nt^n+\cdots\right).
\]

Therefore, the limit augmentation $\mu_1=[\cc;\,\phi,(1,0)]$ 
satisfies $\mu_0<\mu_1< w_G$. \e

 This valuation $\mu_1$ is residue-transcendental and $\phi$ is a key polynomial of minimal degree for $\mu_1$, which is not irreducible in $ \khx$.
We have $Q_{\cc,\phi}=x-2i$.

In order to reach $w_G$, we need to consider another limit augmentation, with respect to the continuous family:
\[
\cc_1=\left(\eta_n\right)_{n\ge0},\qquad \eta_n=[\cc;Q_n,w_G(Q_n)]\qquad Q_n:=\phi+b_n,
\]
where $b_n\in K$ are taken to be
\[
b_0=0,\qquad  b_n=2\left(t+j_2t^2+\cdots+j_nt^n\right),\quad n\ge1.
\]
One can check that $g$ is a limit key polynomial for $\cc_1$, and the MLV of $w_G$  consists of two consecutive limit augmentations:
$$
v\ \stackrel{x,(0,0)}\lra\ \mu_0 \ \stackrel{\phi,(1,0)}\lra\ \mu_1\ \stackrel{g,\infty}\lra w_G.
$$ 

\end{example}

\end{document}